\documentclass[10pt]{amsart}

\usepackage{ amsmath,amssymb,verbatim,graphicx,amsthm}
\usepackage[top=1in, bottom=1in, left=1in, right=1in]{geometry}
\usepackage{cite}
\usepackage[sc]{mathpazo}
\usepackage[T1]{fontenc}

\newcommand{\vertiii}[1]{{\left\vert\kern-0.25ex\left\vert\kern-0.25ex\left\vert #1 
    \right\vert\kern-0.25ex\right\vert\kern-0.25ex\right\vert}}

\newcommand{\RR}{\mathbb{R}}

\newcommand{\ZZ}{\mathbb{Z}}
\newcommand{\QQ}{\mathbb{Q}}
\newcommand{\NN}{\mathbb{N}}

\newcommand{\la}{\left\langle}
\newcommand{\ra}{\right\rangle}

\newcommand{\expec}{\mathbb{E}}

\newcommand{\ind}{\mathbb{1}}
\newcommand{\Id}{\text{Id}}

\theoremstyle{plain}

\newtheorem{theorem}{Theorem}[section]

\newtheorem{lemma}[theorem]{Lemma}
\newtheorem{mydef}[theorem]{Definition}

\usepackage{setspace}
\begin{document}

\title{Pointwise characteristic factors for Wiener-Wintner double recurrence theorem}

\author{Idris Assani}
\address{Department of Mathematics, The University of North Carolina at Chapel Hill, 
Chapel Hill, NC 27599}
\email{assani@math.unc.edu}
\urladdr{http://www.unc.edu/math/Faculty/assani/} % Delete if not wanted.

\author{David Duncan}
\address{Department of Mathematics \& Statistics, Coastal Carolina University, 
Conway, SC 29528}
\email{dduncan@coastal.edu}
\urladdr{http://www.coastal.edu/math/faculty/duncan.html} % Delete if not wanted.

\author{Ryo Moore}
\address{Department of Mathematics, The University of North Carolina at Chapel Hill, 
Chapel Hill, NC 27599}
\email{ryom@live.unc.edu}
%\urladdr{http://www.unc.edu/math/Faculty/assani/} % Delete if not wanted.

\begin{abstract}
In this paper, we extend Bourgain's double recurrence result to the Wiener-Wintner averages. Let $(X, \mathcal{F}, \mu, T)$ be a standard ergodic system. We will show that for any $f_1, f_2 \in L^\infty(X)$, the double recurrence Wiener-Wintner average
\[ \frac{1}{N} \sum_{n=1}^N f_1(T^{an}x)f_2(T^{bn}x) e^{2\pi i n t} \]
converges off a single null set of $X$ independent of $t$ as $N \to \infty$. Furthermore, we will show a uniform Wiener-Wintner double recurrence result: If either $f_1$ or $f_2$ belongs to the orthogonal complement of the Conze-Lesigne factor, then there exists a set of full measure such that the supremum on $t$ of the absolute value of the averages above converges to $0$.
\end{abstract}

\maketitle

\section{Historical Background}
In 1990, Bourgain proved the result on double recurrence \cite{BourgainDR}, which is stated as follows:
\begin{theorem}[Bourgain \cite{BourgainDR}]
Let $(X, \mathcal{F}, \mu, T)$ be an ergodic system, and $T_1$, $T_2$ be powers of $T$. Then, for $f_1, f_2 \in L^\infty(\mu)$,
\[ \lim_{N \to \infty} \frac{1}{N} \sum_{n=1}^N f_1(T_1^nx) f_2(T_2^nx) \]
converges for $\mu$-a.e. $x \in X$.
\end{theorem}
In \cite{BourgainDR}, the theorem was proven for the case $T_1 = T = T_2^{-1}$. Bourgain's proof relies on the uniform Wiener Wintner theorem, which is stated as follows (see, for example, \cite{AssaniWWET} for a proof):
\begin{theorem}\label{uniformWW}
Let $(X, \mathcal{F}, \mu, T)$ be an ergodic system, and let $f$ be a function in the orthogonal complement of the Kronecker factor of $(X, T)$. Then there exists a set of full measure $X_f$ such that
\[ \limsup_{N \to \infty} \sup_{t \in \RR} \left| \frac{1}{N} \sum_{n=1}^N f(T^nx)e^{2\pi i n t} \right| = 0 \]
for all $x \in X_f$.
\end{theorem}

In 2001, the second author worked on an extended result of Bourgain in his Ph. D. thesis \cite{duncan}, and proved the double recurrence Wiener Wintner result for the case when $T$ is totally ergodic (i.e. $T^a$ is ergodic for any $a \in \ZZ$).
\begin{theorem}\label{Duncan}Let $(X, \mathcal{F}, \mu, T)$ be a standard ergodic dynamical system (i.e. $X$ is a compact metrizable space, $\mathcal{F}$ is a Borelian sigma-algebra, $\mu$ is a probability Borel measure, and $T$ is a self-homeomorphism). Suppose $f_1$ and $f_2$ belong to $L^2(X)$. Let $\mathcal{CL}$ be the maximal isometric extension of the Kronecker factor of $T$. Let 
\[W_N(f_1, f_2,x ,t) = \frac{1}{N} \sum_{n=0}^{N-1} f_1(T^{an}x)f_2(T^{bn}x) e^{2\pi i n t} \]
\begin{enumerate}
\item (Double Uniform Wiener-Wintner Theorem) If either $f_1$ or $f_2$ belongs to $\mathcal{CL}^\perp$, then there exists a set of full measure $X_{f_1 \otimes f_2}$ such that for all $x \in X_{f_1 \otimes f_2}$,
\[ \limsup_{N \to \infty} \sup_{t \in \RR} \left|W_N(f_1, f_2, x, t) \right| = 0 \]
\item (General Convergence) If $f_1, f_2 \in \mathcal{CL}$, then $W_N(f_1, f_2, x, t)$ converges for $\mu$-a.e. $x \in X$ for all $t \in \RR$, provided that the cocycle associated with $\mathcal{CL}$ is affine.
\end{enumerate}
%Furthermore, the above Wiener-Wintner averages converges in the case $f_1, f_2 \in \mathcal{C}$, and the cocycle associated with $\mathcal{C}$ is affine.
\end{theorem}
Theorem $\ref{Duncan}$ was proved in several stages. For (1), first one identifies the pointwise limit of the double recurrence averages as an integral with respect to a particular Borel measure (disintegration). Then one uses Wiener's lemma on the continuity of spectral measures and van der Corput's inequality to show that the double recurrence average converges to $0$. For the second part, one first shows that the total ergodicity of $T$ asserts that $\mathcal{CL}$ for every integer power of $T$ are the same, which allows one to assume that both functions lie in the same factor of $L^2(X, \mu)$. Furthermore, the assumption that the measurable cocycle associated with $\mathcal{CL}$ is affine allows one to use the homomorphism property to simplify the computations. 

A little was known about characteristic factors back then, especially for pointwise convergence. Originaly in \cite{duncan}, the factor $\mathcal{CL}$ is referred to as "Conze-Lesigne" factor, as they first appeared in series of work by J.-P. Conze and E. Lesigne (see, for example, \cite{CL84, CL87} for details), and named so by D. Rudolph \cite{Rudolph}. But with the work of B. Host and B. Kra in \cite{HostKraNEA}, the definition of the Conze-Lesigne facter has been updated when the Host-Kra-Ziegler factors emerged in 2005. It is noted that the updated Conze-Lesigne factor $\mathcal{Z}_2$, the second Host-Kra-Ziegler factor, is smaller than $\mathcal{CL}$, so more work is needed to prove the uniform double recurrence Wiener-Wintner theorem for the case either $f_1, f_2 \in \mathcal{Z}_2^\perp$ since $  \mathcal{CL}^\perp \subset \mathcal{Z}_2^\perp$.

\section{Introduction}
In this paper, we will prove the uniform Wiener-Wintner result for the case $f_1 \in \mathcal{Z}_2^\perp$ using the seminorms that characterize these related factors. These characteristic factors and seminorms were developed in the work of B. Host and B. Kra \cite{HostKraNEA} (these characteristic factors were also developed independently by T. Ziegler \cite{Ziegler} without the use of seminorms). 

\begin{mydef}\label{HostKraSeminorms} Let $(X, \mathcal{F}, \mu, T)$ be an ergodic dynamical system on a probability measure space. The factors $\mathcal{Z}_k$ are defined in terms of seminorms as follows.
\begin{itemize}
	\item The factor $\mathcal{Z}_0$ is the trivial $\sigma$-algebra.
	\item The factor $\mathcal{Z}_1$ can be characterized by the seminorm $\vertiii \cdot_2$ where
	\[ \vertiii f _2^4 = \lim_{H \to \infty} \frac{1}{H} \sum_{h=1}^H \left| \int f \cdot f \circ T^h \, d\mu \right|^2 \]
	i.e. A function $f$ belongs to $\mathcal{Z}_1^\perp$ if and only if $\vertiii{ f }_2 = 0$.
	\item The factor $\mathcal{Z}_2$ is the Conze-Lesigne factor. Functions in this factor are characterized by the seminorm $\vertiii \cdot _3$ such that
	\[\vertiii f _3^8 = \lim_{H \to \infty} \frac{1}{H} \sum_{h=1}^H \vertiii{ f \cdot f \circ T^h }^4_2 \: , \]
	i.e. A function $f$ belongs to $\mathcal{Z}_2^\perp$ if and only if $\vertiii f _3 = 0$.
	\item More generally, for each positive integer $k$, we have
	\[ \vertiii f _{k+1}^{2^{k+1}} = \lim_{H \to \infty} \frac{1}{H} \sum_{h=1}^H \vertiii{ f \cdot f \circ T^h }_k^{2^k} \]
	with the condition that $f$ belongs to $\mathcal{Z}_{k-1}^\perp$ if and only if $\vertiii f _k = 0$.
\end{itemize}
\end{mydef}

Note that these are similar to the seminorms introduced by W. T. Gowers in \cite{Gowers}. In parcitular, $\displaystyle \| f \|_{U^k(X)}^{2^k} = \vertiii{ f }_{k}^{2^k}$. In this paper, we chose use the notation $\vertiii{ \cdot }_{k}^{2^k}$ merely for the sake of readability.

In 2012, the first author and K. Presser published an update \cite{AssaniPresser} of their earlier unpublished work \cite{AssaniPresser03} on characteristic factors and the multiterm return times theorem. 
\begin{mydef}
Let $(X, \mathcal{F}, \mu, T)$ be an ergodic dynamical system on a probability measure space. We define factors $\mathcal{A}_k$ in the following inductive way.

\begin{itemize}
	\item The factor $\mathcal{A}_0$ is the trivial $\sigma$-algebra $\left\{ X, \emptyset \right\}$.
	
	\item The factor $\mathcal{A}_1$ is the Kronecker factor of $T$. We denote $N_1(f) = \| \expec(f | \mathcal{A}_1) \|_2$.
	
	\item For $k \geq 1$, the factor $\mathcal{A}_{k+1}$ is characterized by the following: A function $f \in \mathcal{A}^\perp_{k+1}$ if and only if
	\[ N_{k+1}(f)^4 := \lim_H \frac{1}{H} \sum_{h=1}^H \| \expec(f \cdot f \circ T^h | \mathcal{A}_k) \|_2^2 = 0. \]
\end{itemize}
\end{mydef}
%Note that $\mathcal{A}_1$ is the Kronecker factor of our ergodic transformation $T$ because
%\[ N_1(f)^4 = \lim_H \frac{1}{H} \sum_{h=1}^H \| \expec(f \cdot f \circ T^h | \mathcal{A}_0) \|_2^2 = \lim_H \frac{1}{H} \sum_{h=1}^H \left| \int f \cdot f \circ T^h \, d\mu \right|^2 = \lim_H \frac{1}{H} \sum_{h=1}^H | \hat{\sigma}_f(h) |^2\]
%where $\sigma_f$ is the spectral measure of $f$ with respect to measure $\mu$ and the transformation $T$. By Wiener's theorem, the right hand side equals $0$ if and only if $\sigma_f$ is continuous, which is equivalent of saying $f$ belongs to the orthogonal complement of the Kronecker factor of $T$. 
It was proven that the quantities $N_k(f)$ are well-defined in \cite{AssaniCF}, and they characterize factors $\mathcal{A}_k$ of $T$ which are successive maximal isometric extensions. These successive factors turned out to be the $k$-step distal factors introduced by H. Furstenberg in \cite{FurstenbergDistal}.

In \cite{AssaniPresser}, it was shown that given an ergodic system $(X, \mathcal{F}, \mu, T)$ and $f_1 \in L^\infty(\mu)$, there exists a set of full measure $X_f$ such that for any $x \in X_f$ and for any measure-preserving system $(Y, \mathcal{G}, \nu, S)$ and $f_2 \in L^\infty(\nu)$ such that $\| f_2 \|_{L^\infty(\nu)} \leq 1$, the average
\[ \limsup_{N \to \infty} \sup_{t \in \RR} \left| \frac{1}{N} \sum_{n=1}^N f_1(T^n x) f_2(S^ny) e^{2\pi i n t} \right| \leq CN_3(f_1)^2 \]
converges for $\nu$-a.e. $y \in Y$ for some absolute constant $C$ independent of $f_1, f_2, S$, and $y$.

It is known that $\mathcal{Z}_k \subset \mathcal{A}_k$ (in fact, $\mathcal{Z}_0$ equals $\mathcal{A}_0$ and $\mathcal{Z}_1$ equals $\mathcal{A}_1$, but $\mathcal{Z}_2 \subsetneq \mathcal{A}_2$), so $\mathcal{Z}_k^\perp \supset \mathcal{A}_k^\perp$. In \cite{AssaniPresser}, it was proven that $\mathcal{Z}_k$ and $\mathcal{A}_k$ are both pointwise characteristic for the $k$-term return times averages.

In this paper, we will update Theorem $\ref{Duncan}$ in the following ways:
\begin{itemize}
	\item We will only assume that $T$ is ergodic, rather than totally ergodic.
	\item We will show that $\mathcal{Z}_2$ (and $\mathcal{A}_2$) is a characteristic factor for this Wiener-Wintner average, i.e. We will prove the uniform double Wiener-Wintner result for the case either $f_1 \in \mathcal{Z}_2^\perp$ or $f_2 \in \mathcal{Z}_2^\perp$ rather than $\mathcal{CL}^\perp$.
	\item We will show that the convergence holds in general for case $f_1, f_2 \in \mathcal{Z}_2$.
\end{itemize}
In other words, we will prove the following:
\begin{theorem}\label{MainResult}Let $(X, \mathcal{F}, \mu, T)$ be a standard ergodic dynamical system, and $f_1, f_2 \in L^2(X)$. Let 
\[W_N(f_1, f_2,x ,t) = \frac{1}{N} \sum_{n=0}^{N-1} f_1(T^{an}x)f_2(T^{bn}x) e^{2\pi i n t}. \]
\begin{enumerate}
\item (Double Uniform Wiener-Wintner Theorem) If either $f_1$ or $f_2$ belongs to $\mathcal{Z}_2^\perp$, then there exists a set of full measure $X_{f_1 \otimes f_2}$ such that for all $x \in X_{f_1 \otimes f_2}$,
\begin{equation*}\label{UniformDWW} \limsup_{N \to \infty} \sup_{t \in \RR} \left|W_N(f_1, f_2, x, t) \right| = 0 \end{equation*}
\item (General Convergence) If $f_1, f_2 \in \mathcal{Z}_2$, then for $\mu$-a.e. $x \in X$, $W_N(f_1, f_2, x, t)$ converges for all $t \in \RR$.
\end{enumerate}
\end{theorem}

We will use Bourgain's double recurrence theorem and the seminorms mentioned above to prove (1). Some useful inequalities are introduced in section $\ref{sec:Inequalities}$. We will first show that (1) holds for the case $f_1 \in \mathcal{Z}_2^\perp$ when  $a = 1$ and $b = 2$ in section $\ref{sec:CL}$. Complication arises when $|b - a| > 1$, and we will prove the case for general $a, b \in \ZZ$ in section $\ref{sec:general}$. In section $\ref{sec:maxIsom}$, we will prove that there is a pointwise estimate for the limit of the double recurrence Wiener Wintner average using  for the case $a = 1$ and $b = 2$ using the seminorm $N_2( \cdot )$. Finally, in section $\ref{sec:CVinZ2}$, we will show (2) of Theorem $\ref{MainResult}$ using Leibman's convergence result in  \cite{Leibman}.

Throughout this paper, we will without loss of generality assume that the functions $f_1$ and $f_2$ are real-valued, and $\|f_1\|_\infty, \|f_2
\|_\infty \leq 1$, unless specified otherwise. Note that this implies that for any sub-sigma algebra $\mathcal{G}$ of $\mathcal{F}$, $\|\expec(f_i | \mathcal{G})\|_\infty \leq 1$ for both $i = 1, 2$.

\section{Some Inequalities}\label{sec:Inequalities}
Throughout this paper, we will refer to the following inequalities repetitively. The proofs are given in the appendix.

One of the key ingredients of our proofs is the van der Corput's inequality, which is the following.

\begin{lemma}[van der Corput] \label{lem-vdc}
If $(a_n)$ is a sequence of complex numbers and if $H$ is an integer between $0$ and $N-1$, then
\begin{align}\label{vdc}
\left| \frac{1}{N} \sum_{n=0}^{N-1} a_n \right|^2 
&\leq \frac{N+H}{N^2(H+1)} \sum_{n=0}^{N-1} |a_n|^2 \\
&+ \frac{2(N+H)}{N^2(H+1)^2} \sum_{h=1}^H (H + 1 - h) \Re \left( \sum_{n=0}^{N-h-1} a_n \overline{a}_{n+h} \right). \nonumber
\end{align}
\end{lemma}

The following inequalities can be derived directly from ($\ref{vdc}$).

\begin{lemma}\label{lem-vcd-ww}
\begin{itemize}
\item There exists an absolute constant $C$ such that for any sequence of complex numbers $( a_n )$ such that $\sup_n|a_n| \leq 1$ and any positive integer N, we have
\begin{equation}\label{vcd-lim}
\limsup_{N \to \infty} \left|\frac{1}{N} \sum_{n=1}^N a_n \right|^2 \leq \frac{C}{H} + \frac{C}{(H+1)^2} \sum_{h=1}^H (H + 1 - h) \Re \left( \limsup_{N \to \infty} \frac{1}{N} \sum_{n=1}^{N} a_n \overline{a}_{n+h} \right)
\end{equation}
for any $H \in \NN$.
\item There exists an absolute constant $C$ such that for any sequence of complex numbers $( a_n )$ such that $\sup_n|a_n| \leq 1$ and any positive integer N, we have
\begin{equation}\label{vcd-ww} \sup_{t \in \RR} \left| \frac{1}{N} \sum_{n=1}^N a_n e^{2\pi i n t} \right|^2 \leq \frac{C}{H} + \frac{C}{H} \sum_{h=1}^H \left| \frac{1}{N} \sum_{n=1}^{N-h} a_n \overline{a}_{n+h} \right| \end{equation}
for $1 \leq H \leq N$.
\item There exists an absolute constant $C$ such that for any sequence of complex numbers $( a_n )$ such that $\sup_n|a_n| \leq 1$ and any positive integer N, we have
\begin{equation}\label{vcd-ww-lim} \limsup_{N \to \infty} \sup_{t \in \RR} \left| \frac{1}{N} \sum_{n=1}^N a_n e^{2\pi i n t} \right|^2 \leq \frac{C}{H} +  \frac{C}{H} \sum_{h=1}^H  \limsup_{N \to \infty} \left| \frac{1}{N} \sum_{n=1}^{N} a_n \overline{a}_{n+h} \right|
\end{equation}
for all $H \in \NN$.
\end{itemize}
\end{lemma}

%\begin{equation}\label{vcd-ww-lim} \limsup_{N \to \infty} \sup_{t \in \RR} \left| \frac{1}{N} \sum_{n=1}^N a_n e^{2\pi i n t} \right|^2 \leq \frac{C}{H} +  \frac{C}{H} \sum_{h=1}^H  \limsup_{N \to \infty} \left| \frac{1}{N} \sum_{n=1}^N a_n \overline{a}_{n+h} \right|
%\end{equation}

The following inequality is sometimes known as the reverse Fatou lemma.

\begin{lemma}\label{lem-revFatou}
Let $(X, \mathcal{F}, \mu)$ be a measure space. Suppose $(f_n)$ is a sequence of integrable, real-valued functions such that $\sup_n f_n \leq F$ for some integrable function $F$. Then
\begin{equation}\label{revFatou}
\limsup_{n \to \infty} \int f_n \, d\mu \leq \int \limsup_{n \to \infty} f_n \, d\mu
\end{equation}
\end{lemma}

Finally, here is an inequality that can be used to control the average along the cubes. This inequality is similar to Lemma 5 in \cite{AssaniCubes}.

\begin{lemma}\label{lem-cubes-variant}
Let $a_n$, $b_n$, and $c_n$, $n \in \NN$ be three complex-valued sequences, norm of each bounded above by $1$. Then for each positive integer $N$,
\begin{align}
\label{cubes-variant}\left| \frac{1}{N(N+1)^2} \sum_{m, n = 0}^{N-1} (N + 1 - m) a_n \cdot b_m \cdot c_{n+m} \right|^2
\leq \sup_t \left| \frac{1}{N} \sum_{m'=1}^{2(N-1)} c_{m'} e^{2\pi i m' t} \right|^2
\end{align}
\end{lemma}
\section{When $f_1 \in \mathcal{Z}_2^\perp$, $a = 1$, $b = 2$}\label{sec:CL}

In this section, we will prove the uniform Wiener Wintner theorem for the case when $f_1$ belongs to $\mathcal{Z}_2^\perp$. We will prove this special case since the fact that $|b - a| = 1$ simplifies the proofs tremendously, since for any $f, g \in L^2(\mu)$,
\[ \int f(Tx)g(T^2x) d\mu(x) = \int f(x)g(Tx) d\mu(x). \] 

\begin{theorem}\label{CL} Let $(X, \mathcal{F}, \mu, T)$ be an ergodic dynamical system, and $f_1, f_2 \in L^\infty(X)$, and $\| f_i \|_\infty \leq 1$ for both $i = 1, 2$. If $f_1 \in \mathcal{Z}_2^\perp$, then
\[ \limsup_{N \to \infty} \sup_{t \in \RR} \left| \frac{1}{N} \sum_{n=1}^N f_1(T^nx) f_2(T^{2n}x) e^{2\pi i n t} \right| = 0 \]
for $\mu$-a.e. $x \in X$.
\end{theorem}

We will present two proofs; the first one is more direct and concise than the second one, while the second one will be similar to the proof for the general case (when $a, b \in \ZZ$).

In the first proof, we first find the upper bound for the limit supremum for the $L^2$-norm of the average of the sequence $G_1(T^nx)G_2(T^{2n}x)$ for any $G_1, G_2 \in L^\infty(\mu)$ using $(\ref{vcd-lim})$---this upper bound turns out to be a constant multiple of $\vertiii{  f_1 }_2^2$. Then we use this upper bound as well as the double recurrence theorem and inequality $(\ref{vcd-ww-lim})$ to find the upper bound for the norm of the limit supremum of the average of the sequence $f_1(T^nx)f_2(T^{2n}x)e^{2\pi i n t}$, which turns out to be a constant multiple of $\vertiii{ f_1 }_3^2$.

For the second proof, we first apply the inequality $(\ref{vcd-ww})$ by setting $a_n = f_1(T^nx)f_2(T^{2n}x)$ pointwise, and then apply the inequality $(\ref{vdc})$ to the new average. After noticing that the average after the VDC trick converges a.e. (by the double recurrence theorem), we take the limit supremum of the first average and integrate. Using the ergodic decomposition, Wiener's lemma, the inequality ($\ref{cubes-variant}$), and an inequality found in \cite{AssaniPresser}, we can conclude that the original average converges to zero $\mu$-a.e.

\begin{proof}[First proof of Theorem $\ref{CL}$]
First, we would like to show that for any two functions $G_1, G_2 \in L^\infty(\mu)$, where $\|G_i\|_{L^\infty(\mu)} \leq 1$ for $i = 1, 2$, the following estimate holds:
\begin{equation}\label{Z2ineq}
\limsup_{N \to \infty} \int \left| \frac{1}{N} \sum_{n=1}^N G_1(T^nx) G_2(T^{2n}x) \right|^2 d\mu \leq C \vertiii {G_1} _2^2.
\end{equation}
First we apply the inequality ($\ref{vcd-lim}$) by setting $a_n = G_1(T^nx) G_2(T^{2n}x)$ pointwise. Then we integrate both sides and apply $(\ref{revFatou})$ to obtain
\begin{align} & \label{firstIntegral}\limsup_{N \to \infty} \int \left| \frac{1}{N} \sum_{n=1}^N G_1(T^nx) G_2(T^{2n}x) \right|^2 d\mu  \leq  \int \limsup_{N \to \infty} \left| \frac{1}{N} \sum_{n=1}^N G_1(T^nx) G_2(T^{2n}x) \right|^2 d\mu\\
& \nonumber \leq \frac{C}{H} + \frac{C}{(H+1)^2} \sum_{h=1}^{H} (H - h + 1) \Re \left(\int \limsup_{N \to \infty} \frac{1}{N} \sum_{n=1}^{N} (G_1 \cdot G_1 \circ T^h)(T^nx)(G_2 \cdot G_2 \circ T^{2h})(T^{2n}x)  d\mu\right).
\end{align} 
Note that $\displaystyle \lim_{N \to \infty} \frac{1}{N} \sum_{n=1}^{N} (G_1 \cdot G_1 \circ T^h)(T^nx)(G_2 \cdot G_2 \circ T^{2h})(T^{2n}x)$ exists for $\mu$-a.e. $x \in X$ by the double recurrence theorem. Hence, the dominated convergence theorem tells us that
\begin{align*} & \label{firstIntegral}\limsup_{N \to \infty} \int \left| \frac{1}{N} \sum_{n=1}^N G_1(T^nx) G_2(T^{2n}x) \right|^2 d\mu  \\
& \nonumber \leq \frac{C}{H} + \frac{C}{(H+1)^2} \sum_{h=1}^{H} (H - h + 1) \Re \left(\lim_{N \to \infty}\int  \frac{1}{N} \sum_{n=1}^{N} (G_1 \cdot G_1 \circ T^h)(T^nx)(G_2 \cdot G_2 \circ T^{2h})(T^{2n}x)  d\mu\right).
\end{align*} 
Using the fact that $T$ is measure-preserving, we can apply the mean ergodic theorem to obtain 
\begin{align*}
& \limsup_{N \to \infty} \int \left| \frac{1}{N} \sum_{n=1}^N G_1(T^nx) G_2(T^{2n}x) \right|^2 d\mu  \\
&\leq \frac{C}{H} + \frac{C}{H} \sum_{h=1}^{H} \left| \lim_{N \to \infty} \int (G_1 \cdot G_1 \circ T^h)(x) \frac{1}{N}\sum_{n=1}^{N} (G_2 \cdot G_2 \circ T^{2h})(T^{n}x)  d\mu\right|\\
&= \frac{C}{H} + \frac{C}{H} \sum_{h=1}^{H} \left| \int (G_1 \cdot G_1 \circ T^h)(x) d\mu(x) \int (G_2 \cdot G_2 \circ T^{2h})(y)  d\mu(y)\right|.
\end{align*}
Because $\| G_2 \|_{L^\infty(\mu)} \leq 1$, we know that $\displaystyle \left| \int (G_2 \cdot G_2 \circ T^{2h})(y)  d\mu(y)\right| \leq 1$. Hence, we can apply the Cauchy-Schwarz inequality and let $H \to \infty$ to obtain $(\ref{Z2ineq})$.

Now we are ready to prove the theorem. Our goal is to show that there exists a universal constant $C$ such that
\[ \int \limsup_{N \to \infty} \sup_{t \in \RR} \left| \frac{1}{N} \sum_{n=1}^N f_1(T^nx)f_2(T^{2n}x)e^{2\pi i n t} \right|^2 d\mu \leq C \vertiii {f_1} _3^2. \]
By using the inequality ($\ref{vcd-ww-lim}$) by setting $a_n = f_1(T^nx)f_2(T^{2n}x)$ pointwise, we obtain
\begin{equation*}
\limsup_{N \to \infty}\sup_{t \in \RR} \left| \frac{1}{N} \sum_{n=1}^N f_1(T^nx) f_2(T^{2n}x) e^{2\pi int} \right|^2  \leq \frac{C}{H}+\frac{C}{H} \sum_{h=1}^H  \lim_{N \to \infty} \left| \frac{1}{N}\sum_{n=1}^{N} F_{1, h}(T^nx) F_{2, h}(T^{2n}x) \right|,
\end{equation*}
where $F_{1, h}(x) = f_1(x) f_1 \circ T^h(x)$, and $F_{2, h}(x) = f_2(x) f_2 \circ T^{2h}(x)$, and the limit on the right hand side exists by the double recurrence theorem.
We take the integral on both sides of the inequality above, and apply the dominated convergence theorem and the Cauchy-Schwarz inequality to obtain
\begin{equation*}
\int \limsup_{N \to \infty}\sup_{t \in \RR} \left| \frac{1}{N} \sum_{n=1}^N f_1(T^nx) f_2(T^{2n}x) e^{2\pi int} \right|^2 d\mu  \leq \frac{C}{H}+\frac{C}{H} \sum_{h=1}^H  \left( \lim_{N \to \infty} \int \left| \frac{1}{N}\sum_{n=1}^{N} F_{1, h}(T^nx) F_{2, h}(T^{2n}x) \right|^2 d\mu\right)^{1/2}.
\end{equation*}
Apply ($\ref{Z2ineq}$) by setting $G_1 = F_{1, h}$ and $G_2 = F_{2, h}$ on the right hand side while letting $H \to \infty$, we obtain the desired upper bound; we have
\begin{align*}
&\int \limsup_{N \to \infty} \sup_{t \in \RR} \left| \frac{1}{N} \sum_{n=1}^N f_1(T^nx) f_2(T^{2n}x) e^{2\pi int} \right|^2 d\mu \leq \limsup_{H \to \infty} \frac{C}{H} \sum_{h=1}^H \vertiii{ f_1 \cdot f_1 \circ T^h }_2\\
&\leq C\left( \lim_{H \to \infty} \frac{1}{H} \sum_{h=1}^H \vertiii{f_1 \cdot f_1 \circ T^h }_2^4 \right)^{1/4} = C \vertiii{f_1}_3^2.
\end{align*}
Since $f_1$ belongs to $\mathcal{Z}_2^\perp$, we know that $\vertiii{ f_1 }_3 = 0$, which completes the proof.
\end{proof}

\begin{proof}[Second Proof]
We denote $F_{1, h}(x) = f_1(x) f_1 \circ T^h(x)$, and $F_{2, h}(x) = f_2(x) f_2 \circ T^{2h}(x)$. We apply the inequality $(\ref{vcd-ww-lim})$ by setting $a_n = f_1(T^nx)f_2(T^{2n}x)$ pointwise, and the Cauchy-Schwarz inequality to obtain the following estimate for any $H \in \NN$:
\begin{align}
\label{secondIntegral} \limsup_{N \to \infty}\sup_{t \in \RR} \left| \frac{1}{N} \sum_{n=1}^N f_1(T^nx)f_2(T^{2n}x) e^{2\pi i n t} \right|^2
& \leq \frac{C}{H} + \frac{C}{H} \sum_{h = 1}^H \limsup_{N \to \infty} \left| \frac{1}{N} \sum_{n=1}^{N}  F_{1, h}(T^nx) F_{2, h} (T^{2n}x) \right|\\
&\nonumber \leq \frac{C}{H} + \left(\frac{C}{H}  \sum_{h = 1}^H \limsup_{N \to \infty} \left| \frac{1}{N} \sum_{n=1}^{N} F_{1, h}(T^nx) F_{2, h} (T^{2n}x) \right|^2 \right)^{1/2}
\end{align}
%Note that the third term on the second line vanishes as $N \to \infty$, so we can carefully ignore this term. For the second term, we apply the Cauchy-Schwarz inequality, and then we apply the $(\ref{vdc})$ on the average inside the absolute value by setting so that for any $0 < K < N$, we have the following upper bound:
Now we apply the inequality $(\ref{vcd-lim})$ on the average inside the absolute value by setting $a_n = F_{1, h}(T^nx)F_{2, h}(T^{2n}x)$ pointwise so that for any $K \in \NN$, we have 
\begin{align}
&\nonumber \limsup_{N \to \infty} \left| \frac{1}{N} \sum_{n=1}^{N} F_{1, h}(T^nx)F_{2, h}(T^{2n}x) \right|^2 \\
&\label{secondVDC} \leq \frac{C}{K} + \frac{C}{(K+1)^2} \sum_{k=1}^K (K - k+1) \left( \limsup_{N \to \infty}\frac{1}{N} \sum_{n=1}^{N} \left(F_{1, h} \cdot F_{1, h} \circ T^k \right)(T^nx)   \left(F_{2, h} \cdot F_{2, h} \circ T^{2k} \right)(T^{2n}x) \right).
\end{align}
Note that the average
\[ \frac{1}{N} \sum_{n=1}^{N}  \left(F_{1, h} \cdot F_{1, h} \circ T^k \right)(T^nx)   \left(F_{2, h} \cdot F_{2, h} \circ T^{2k} \right)(T^{2n}x) \]
converges as $N \to \infty$ by the double recurrence theorem. Now we combine $(\ref{secondIntegral})$ and $(\ref{secondVDC})$, integrate both sides, and apply the H\"{o}lder's inequality as well as the dominated convergence theorem to obtain

\begin{align*} 
& \int \limsup_{N \to \infty} \sup_{t \in \RR} \left| \frac{1}{N} \sum_{n=1}^N f_1(T^nx) f_2(T^{2n}x) e^{2\pi i n t} \right|^2 d\mu \\
&\leq \frac{C}{H} + \int \left( \frac{C}{H}\sum_{h = 1}^H \limsup_{N \to \infty} \left| \frac{1}{N} \sum_{n=1}^{N} F_{1, h}(T^nx) F_{2, h}(T^{2n}x) \right|^2 \right)^{1/2} d\mu\\
&\leq \frac{C}{H} + \left(  \frac{C}{H}\sum_{h = 1}^H \int  \limsup_{N \to \infty}\left| \frac{1}{N} \sum_{n=1}^{N}F_{1, h}(T^nx)F_{2, h}(T^{2n}x) \right|^2 d\mu \right)^{1/2}  \\
&\leq \frac{C}{H} + \left( \frac{C}{H}\sum_{h=1}^H \left( \frac{C}{K} + \frac{C}{(K+1)^2} \sum_{k=1}^K (K+1-k) \right. \right. \\
& \left. \left. \lim_{N \to \infty} \int  \left(  \frac{1}{N} \sum_{n=1}^{N} \left(F_{1, h} \cdot F_{1, h} \circ T^k \right)(T^nx)   \left(F_{2, h} \cdot F_{2, h} \circ T^{2k} \right)(T^{2n}x) \, d\mu\right)\right) \right)^{1/2}.
\end{align*}
Using the fact that $T$ is measure-preserving, we can evaluate the limit in the last term by applying the mean ergodic theorem: We have
\begin{align*}
& \lim_{N \to \infty} \int \frac{1}{N} \sum_{n=1}^{N} \left(F_{1, h} \cdot F_{1, h} \circ T^k \right)(T^nx)   \left(F_{2, h} \cdot F_{2, h} \circ T^{2k} \right)(T^{2n}x)d\mu(x) \\
&=   \lim_{N \to \infty}  \int  \left(F_{1, h} \cdot F_{1, h} \circ T^k \right)(x)   \frac{1}{N} \sum_{n=1}^{N} \left(F_{2, h} \cdot F_{2, h} \circ T^{2k} \right) (T^{n}x) d\mu(x)
\end{align*}
\begin{align*}
&= \iint F_{1, h}(x) (F_{1, h} \circ T^k) (x) F_{2, h}(y)(F_{2, h} \circ T^{2k}) (y) \, d\mu(x) d\mu(y) \\
&= \iint f_1 \otimes f_2(x, y) \,{f_1 \otimes f_2} (U^h(x, y)) \, {f_1 \otimes f_2}(U^k(x, y)) \, f_1 \otimes f_2(U^{h+k}(x, y)) \, d\mu \otimes \mu(x, y),
\end{align*}
where $U = T \otimes T^2$ is a measure preserving transformation on $X^2$. If we take the ergodic decomposition of $\mu \otimes \mu$ with respect to $U$, then the integral becomes
\[ \iint f_1 \otimes f_2(x, y) \, {f_1 \otimes f_2} (U^h(x, y)) \, {f_1 \otimes f_2}(U^k(x, y)) \, f_1 \otimes f_2(U^{h+k}(x, y)) \, d(\mu \otimes \mu)_c(x, y) \, d\mu(c). \]
Let $H = K$. Note that, on the system $(X^2, (\mu \otimes \mu)_c, U)$ for a.e. $c \in X$, the inequality $(\ref{cubes-variant})$ tells us that we have
\begin{align}\label{Cubes} & \left|\frac{1}{H(H+1)^2}  \sum_{h, k = 0}^{H-1} (H + 1 - k) f_1 \otimes f_2 (U^h(x, y)) \, f_1 \otimes f_2 (U^k(x, y))  \, f_1 \otimes f_2 (U^{h+k}(x, y)) \right|^2 \nonumber \\
&\leq  \sup_{t \in \RR} \left| \frac{1}{H} \sum_{h=1}^{2(H-1)} f_1 \otimes f_2 (U^h(x, y)) e^{2\pi iht} \right|^2. \end{align}
Note that the proof is complete if we can show that $\displaystyle \limsup_{H \to \infty} \sup_{t \in \RR} \left| \frac{1}{H} \sum_{h=1}^H f_1 \otimes f_2 (U^h(x, y)) e^{2\pi iht} \right|^2 = 0$ for $(\mu \otimes \mu)_c$-a.e. $(x, y) \in X^2$, for $\mu$-a.e. $c \in X$. In other words, we would like to show that if $f_1$ belongs to $\mathcal{Z}_2^\perp$, then $f_1 \otimes f_2$ belongs to the orthogonal complement of the Kronecker factor with respect to the transformation $U$ and measure $(\mu \otimes \mu)_c$ for $\mu$-a.e. $c \in X$, so that we can apply the uniform Wiener Wintner theorem (Theorem $\ref{uniformWW}$).

To show that this is indeed the case, we first prove the following lemma:

\begin{lemma}\label{spectral}
Suppose $(Y, \mathcal{G}, \nu, U)$ is a measure preserving system, and $f \in L^\infty(X)$ such that
\[ \int \limsup_N \sup_{t \in \RR} \left|\frac{1}{ N} \sum_{n=1}^N f(U^ny) e^{2\pi i n t} \right| d\nu = 0. \]
If $\sigma_f$ is the spectral measure of $f$ with respect to $U$, then
\[ \lim_{N \to \infty} \frac{1}{N} \sum_{n=1}^N | \hat{\sigma}_f(n) |^2 = 0. \]
\end{lemma}

\begin{proof}
By Wiener's lemma, we have
\[ \lim_{N \to \infty} \frac{1}{N} \sum_{n=1}^N | \hat{\sigma}_f(n) |^2 = \sum_{t} | \sigma_f( \left\{ t \right\} ) |^2. \]
Observe that, by the spectral theorem,
\[ |\sigma_f( \left\{ t \right\} )| = \left| \lim_{N \to \infty} \frac{1}{N} \sum_{n=1}^N \hat{\sigma}_f(n) e^{2\pi i n t} \right| = \left| \lim_{N \to \infty} \int f(y)  \frac{1}{N} \sum_{n=1}^N f(U^ny) e^{2\pi int} d\nu(y) \right|. \]
Since $\displaystyle{\lim_{N \to \infty} \frac{1}{N} \sum_{n=1}^N f(U^ny) e^{2\pi int}}$ exists by the Wiener Wintner pointwise ergodic theorem, we can apply the dominated convergence theorem to show that
\begin{align*} 
|\sigma_f( \left\{ t \right\} )| 
&= \left| \int f (y) \lim_{N \to \infty} \frac{1}{N} \sum_{n=1}^N f (U^ny)  e^{2\pi i n t} d\nu(y) \right| \\
&\leq \| f \|_\infty \int \limsup_N \sup_{t \in \RR} \left| \frac{1}{ N} \sum_{n=1}^N f(U^n y) e^{2 \pi i n t} \right| d\nu(y) = 0.
\end{align*}
\end{proof}
The following lemma will complete the proof of this theorem.
\begin{lemma}\label{CLperpKperp}
Suppose $f_1 \in \mathcal{Z}_2^\perp$. Then $f_1 \otimes f_2$  belongs to the orthogonal complement of the Kronecker factor of $U = T \times T^2$ with respect to measure $(\mu \otimes \mu)_c$ for $\mu$-a.e. $c$.
\end{lemma}

\begin{proof}
Equivalently, we would like to show that if $\sigma_{f_1 \otimes f_2}$ is the spectral measure with respect to $U$, we have
\[ \frac{1}{N} \sum_{n=1}^N | \hat{\sigma}_{f_1 \otimes f_2}(n) |^2 \to 0 \]
i.e. we would like to show that $\sigma_{f_1 \otimes f_2}$ is a continuous measure. By lemma $\ref{spectral}$, this can be done by showing
\begin{equation}\label{goal} \int \limsup_{N \to \infty} \sup_{t \in \RR} \left| \frac{1}{N} \sum_{n=1}^N f_1 \otimes f_2(U^n(x, y)) e^{2\pi int} \right| d(\mu \otimes \mu)_c(x, y) = 0 \end{equation}
for $\mu$-a.e. $c \in X$. Because of the ergodic decomposition and Fubini's theorem, we have
\begin{align*}
& \iint \limsup_{N \to \infty} \sup_{t \in \RR} \left| \frac{1}{N} \sum_{n=1}^N f_1 \otimes f_2(U^n(x, y)) e^{2\pi int} \right| \, d (\mu \otimes \mu)_c(x, y) \, d\mu(c) \\
&= \int \left( \int \limsup_{N \to \infty} \sup_{t \in \RR} \left| \frac{1}{N} \sum_{n=1}^N f_1(T^nx) f_2(T^{2n} y) e^{2\pi int} \right| d\mu(y) \right) d\mu(x).
\end{align*}
The inner integral has an upper bound, as it is stated and proved in Lemma 8 of \cite{AssaniPresser}: There exists a set of full measure $X_{f_1}$ such that for all $x \in X_{f_1}$,
\[ \int \limsup_{N \to \infty} \sup_{t \in \RR} \left| \frac{1}{N} \sum_{n=1}^N f_1(T^nx)f_2(T^{2n}y) e^{2\pi int} \right| d\mu(y) \leq C \vertiii{ f_1 }_3^2 = 0 \]
for $\mu$-a.e. $y \in X$. Therefore, 
\begin{equation}\label{ergDecomp} \iint \limsup_{N \to \infty} \sup_{t \in \RR} \left| \frac{1}{N} \sum_{n=1}^N f_1 \otimes f_2(U^n(x, y)) e^{2\pi int} \right| \, d (\mu \otimes \mu)_c(x, y) \, d\mu(c) = 0. \end{equation}
Since 
\[ \int \limsup_{N \to \infty} \sup_{t \in \RR} \left| \frac{1}{N} \sum_{n=1}^N f_1 \otimes f_2(U^n(x, y)) e^{2\pi int} \right| \, d(\mu \otimes \mu)_c(x, y) \]
is $\mu$-a.e. non-negative function of $c$ that is measurable with respect to $\mu$, we can deduce from ($\ref{ergDecomp}$) that $(\ref{goal})$ equals 0.
\end{proof}
Because of lemma $\ref{CLperpKperp}$, we can show that the average $(\ref{Cubes})$ converges to $0$ a.e., which proves theorem $\ref{CL}$.
\end{proof}
\textbf{Remark:} A result similar to Lemma $\ref{CLperpKperp}$ was proven by D. Rudolph in \cite{Rudolph}. In his work, the Conze-Lesigne algebra referred is the maximal isometric extension of the Kronecker factor, which is $\mathcal{CL}$ in Theorem $\ref{Duncan}$ of this paper.

\section{Case $f_1 \in \mathcal{Z}_2^\perp$, when $a, b \in \ZZ$.}\label{sec:general}

Here, we will prove the uniform Wiener Wintner double recurrence property for any $a, b \in \ZZ$, which is stated precisely as follows.
\begin{theorem}\label{generalCL}
Let $(X, \mathcal{F}, \mu, T)$ be an ergodic dynamical system, and $f_1, f_2 \in L^\infty(X)$, and $\| f_2 \|_\infty = 1$. If $f_1 \in \mathcal{Z}_2^\perp$, then
\[ \limsup_{N \to \infty} \sup_{t \in \RR} \left| \frac{1}{N} \sum_{n=1}^N f_1(T^{an}x) f_2(T^{bn}x) e^{2\pi i n t} \right| = 0 \]
for $\mu$-a.e. $x \in X$, and for any pair of integers $a$ and $b$.
\end{theorem}

Unlike the case in Theorem $\ref{CL}$, $T^{b-a}$ may no longer be ergodic. We will first prove various lemmas to overcome this obstacle.

The following lemma will show that for any positive integer $s$, any $T^s$-invariant function $f$ can be expressed in terms of an integral kernel (that does not depend on $f$). The kernel first appeared in the work of H. Furstenberg and B. Weiss (Theorem 2.1 of \cite{FurstenbergWeiss}); we will present a detailed proof here. This kernel will be useful to characterize various conditional expectations.

\begin{lemma}\label{integralKernel}
Let $T$ be an ergodic map, and $s$ be a positive integer. Then there exist a disjoint partition of $T^s$-invariant sets $A_1, \ldots, A_l$ such that every $T^s$-invariant function $f$ can be expressed as an integral with respect to the kernel
\begin{equation}\label{theKernel}K(x, y) = l \sum_{k=1}^l \ind_{A_i}(x) \ind_{A_i}(y). \end{equation}
\end{lemma}

\begin{proof}\footnote{The authors had an opportunity to discuss with B. Weiss about this proof, recently during the Ergodic Theory Workshop at UNC-Chapel Hill (April 3-6, 2014). The proof given here is longer than what Weiss provided, but our proof provides some information about the number $l$.}
If $T^s$ is ergodic, we are done, since $f$ is a constant. If not, suppose $A$ is a $T^s$-invariant subset of $X$ such that $0 < \mu(A) < 1$. Define a function
\[ f_A := \ind_{A} + \ind_{T^{-1}A} + \ind_{T^{-2}A} + \cdots + \ind_{T^{-(s-1)}A} \]
Observe that $f_A$ is $T$-invariant, and since $T$ is ergodic, $f_A$ must be a constant. Therefore, \[ \ind_{A} + \ind_{T^{-1}A} + \ind_{T^{-2}A} + \cdots + \ind_{T^{-(s-1)}A} = \int f_A d\mu = s\mu(A). \]
Note that $f_A \neq 0$, since $\mu(A) \neq 0$. Similarly, $f_A \neq s$, since $\mu(A) \neq 1$. If $f_A = 1$, then for $\mu$-a.e. $x \in X$, $\ind_A + \ind_{T^{-1}A} + \ind_{T^{-2}A} + \cdots + \ind_{T^{-(s-1)}A} = 1$, which implies that $\mu(T^{-i}A \cap T^{-j}A) = 0$ for any $0 \leq i < j \leq s-1$. Hence, $A, T^{-1}A, \ldots, T^{-(s-1)}A$ are disjoint, and furthermore, $\displaystyle \mu(X) = \sum_{k=0}^{s-1} \mu(T^{-k}A) = 1$, so $A, T^{-1}A, \ldots, T^{-(s-1)}A$ is a partition of $X$.

Now we show that $A$ (and similarly, $T^{-1}A, \ldots, T^{-(s-1)}A$) is an atom (of a collection of $T^s$-invariant sets). If $B \subset A$ and $B$ is $T^s$-invariant, then
\[ f_B = \ind_{B} + \ind_{T^{-1}B} + \ind_{T^{-2}B} + \cdots + \ind_{T^{-(s-1)}B} = s\mu(B) \leq s\mu(A) = 1 \]
The above holds only when $\mu(B) = 0$ or $\mu(B) = 1/s = \mu(A)$, which implies that $B = A$ $\mu$-a.e. For $k > 0$, we note that if $B \subset T^{-k}A$ is $T^s$-invariant, then $T^kB \subset A$ is also $T^s$-invariant, so if $\mu(B) \neq 0$, then $\mu(B) = \mu(T^k B) = \mu(A) = \mu(T^{-k}A)$, which proves that $T^{-k}A$ is also an atom for $k > 0$.

If $f$ is a $T^s$-invariant function, then we claim that
\begin{equation}\label{CE} f = \sum_{k=0}^{s-1} \left(\frac{ \int_{T^{-k}A} f d\mu }{\mu(T^{-k}A)} \right) \ind_{T^{-k}A} = s \sum_{k=0}^{s-1} \left( \int_{T^{-k}A} f d\mu\right) \ind_{T^{-k}A}  . \end{equation}
First, we note that $\mathcal{S}$, the $\sigma$-algebra generated by the sets $A, T^{-1}A, \ldots, T^{-(s-1)}A$, is a collection of finite union of sets $A, T^{-1}A, \ldots, T^{-(s-1)}A$. We also know that $f$ is $\mathcal{S}$-measurable, since
\[ \left\{ f > \lambda \right\} = \bigcup_{k=0}^{s-1} \left( \left\{ f > \lambda \right\} \cap T^{-k}A \right), \]
and we note that $\left\{ f > \lambda \right\} \cap T^{-k}A $ is $T^s$-invariant. Since $T^{-k}A$ is an atom for each $k$, we know  $\left\{ f > \lambda \right\} \cap T^{-k}A$ equals either $T^{-k}A$ or the empty set. This implies that $\left\{ f > \lambda \right\} \in \mathcal{S}$.

Since we know that $f$ is $\mathcal{S}$-measurable, we note that $f$ can be expressed as the expression above (a fact regarding conditional expectation). This proves $(\ref{CE})$, and if we denote $T^{-k}A = A_k$, then we have
\[f \circ T^s(x) = f(x) = \int s \sum_{k=0}^{s-1} \ind_{A_k}(y) \ind_{A_k}(x)f(y) d\mu(y) 
               = \int f(y) K(x, y) d\mu(y), \]
which proves the lemma for the case $f_A = 1$.

Now, suppose $f_A = k$ for $2 \leq k \leq s-1$. Let $B = T^{-l_1}A \cap T^{-l_2}A \cap \cdots T^{-l_k}A$, where $0 \leq l_1 < l_2 < \cdots < l_k \leq s-1$, and $\mu(B) > 0$ (we know such $B$ exists since $f_A = k$). Define
\[ f_B = \ind_B + \ind_{T^{-1}B} + \cdots + \ind_{T^{-(s-1)}B} \]
Note that $f_B$ is $T$-invariant, so it must be a constant that equals to $s\mu(B)$. Since $\mu(B) > 0$, we know that $f_B > 0$.

Also, note that each $T^{-j}B$ is disjoint for $0 \leq j \leq s-1$. Assume it is not. Then for some $0 \leq i < j \leq s-1$, there exists $x \in T^{-i}B \cap T^{-j}B$ such that $f_A(x) > k$, which is a contradiction. Therefore, we must have $f_B \leq 1$, and we can conclude that $f_B = 1$. By letting $A_i = T^{-i}B$, we have proved the lemma.
\end{proof}

The next lemma will provide a simple yet useful comparison between $\displaystyle{ \limsup_{H \to \infty} \frac{1}{H} \sum_{h=1}^H \vertiii{ f \cdot f \circ T^{ah} }_k^{2^k}}$ and $\displaystyle{ \vertiii{ f }_{k+1}^{2^{k+1}}}$. 

\begin{lemma}\label{HKPower}
Let $(X, \mathcal{F}, \mu, T)$ be an ergodic dynamical system, and $a \in \ZZ$. Then for any positive integer $k$, we have
\[ \limsup_{H \to \infty} \frac{1}{H} \sum_{h=1}^H \vertiii{ f \cdot f \circ T^{ah} }_k^{2^k} \leq |a| \vertiii{ f }_{k+1}^{2^{k+1}}. \]
\end{lemma}

\begin{proof}
Note that
\[ \frac{1}{H} \sum_{h=1}^H \vertiii{ f \cdot f \circ T^{ah} }_k^{2^k} \leq \frac{1}{H} \sum_{h=1}^{|a|H} \vertiii{ f \cdot f \circ T^{h} }_k^{2^k} = |a|\left( \frac{1}{|a|H} \sum_{h=1}^{|a|H} \vertiii{ f \cdot f \circ T^{h} }_k^{2^k} \right). \]
The sequence $\displaystyle{\left( \frac{1}{|a|H} \sum_{h=1}^{|a|H} \vertiii{ f \cdot f \circ T^{h} }_k^{2^k} \right)_H}$ is a subsequence of $\displaystyle{\left( \frac{1}{H} \sum_{h=1}^H \vertiii{ f \cdot f \circ T^{h} }_k^{2^k} \right)_H}$, which converges to $\vertiii{ f }_{k+1}^{2^{k+1}}$. By taking the limit supremum on both sides of the inequality above, we get
\[ \limsup_{H \to \infty} \frac{1}{H} \sum_{h=1}^H \vertiii{ f \cdot f \circ T^{ah} }_k^{2^k} 
\leq |a| \left( \lim_{H \to \infty} \frac{1}{|a|H} \sum_{h=1}^{|a|H} \vertiii{ f \cdot f \circ T^{h} }_k^{2^k} \right)
= |a| \vertiii{ f }_{k+1}^{2^{k+1}} \]
\end{proof}

The proof of Bourgain's double recurrence theorem in \cite{BourgainDR} relies on the classical uniform Wiener-Wintner theorem, which holds for the case when $T$ is ergodic. Here, we prove the uniform Wiener Wintner theorem that holds for the case when the measure preserving transformation is a power of ergodic map. This allows us to use Bourgain's double recurrence theorem without assuming that $T$ is a totally ergodic map.

\begin{theorem}\label{UniformWWPower}
Let $(X, \mathcal{F}, \mu, T)$ be an ergodic system. Suppose $f \in  \mathcal{Z}_1^\perp$. Then there exists a set of full measure $X_f$ such that for any $x \in X_f$ and for any integer $a$, we have
\[ \limsup_{N \to \infty} \sup_{t \in \RR} \left| \frac{1}{N} \sum_{n=1}^N f(T^{an}x) e^{2\pi i n t} \right| = 0 \]
\end{theorem}
 
\begin{proof}
To show that the uniform convergence holds, we apply the inequality $(\ref{vcd-ww-lim})$ for $a_n = f(T^{an}x)$ pointwise, and use the pointwise ergodic theorem and the Cauchy-Schwarz inequality to obtain
\begin{align} \nonumber \limsup_{N \to \infty}\sup_{t \in \RR}\left| \frac{1}{N} \sum_{n=1}^N f(T^{an}x)e^{2\pi i n t} \right|^2 
&\leq \frac{C}{H} + \frac{C}{H} \sum_{h=1}^H \left| \lim_{N \to \infty} \frac{1}{N} \sum_{n=1}^N (f \cdot f \circ T^{ah})(T^nx) \right| \\
&\nonumber = \frac{C}{H} + \frac{C}{H} \sum_{h=1}^H \left| \expec_a(f \cdot f \circ T^{ah})(x) \right| \\
&\label{condExpec-H}\leq \frac{C}{H} + \left( \frac{C}{H} \sum_{h=1}^H \left| \expec_a(f \cdot f \circ T^{ah})(x)  \right|^2 \right)^{1/2},
\end{align}
where $\expec_a$ is the conditional expectation operator with respect to the sigma-algebra of $T^a$-invariant sets. Let $\gamma_x$ be a measure on $\RR$ such that $\hat{\gamma_x}(h) = \expec_a(f\cdot f \circ T^{ah})(x)$. We would like to show that that $\gamma_x$ is a continuous measure, since that would tells us that the limit above converges to $0$ by Wiener's lemma:
\[ \lim_{H \to \infty} \frac{1}{H} \sum_{h=1}^{H} |\hat{\gamma}_x(h)|^2 = \lim_{H \to \infty} \frac{1}{H} \sum_{h=1}^{H}|\expec_a(f \cdot f \circ  T^{ah}) |^2 = 0. \]
To show this, we use the integral kernel from lemma $\ref{integralKernel}$ so that for some positive integer $l$, we have
\begin{align*}
\expec_a(f \cdot f \circ T^{ah})(x) = l \int \ind_{A_i}(y)f(y)f(T^{ah}y) d\mu(y)
\end{align*}
where $A_i$ is one of the sets of the partition of $X$ given in lemma $\ref{integralKernel}$ such that $x \in A_i$. Set $g(y) = \ind_{A_i}(y)f(y)$. Then we notice that
\[ \int g(y) f(T^{ah}y) d\mu(y) = \hat{\sigma}_{f, g,  T^a}(h),\]
where $\sigma_{f, g, T^a}$ is the spectral measure of the functions $f$ and $g$ with respect to the transformation $T^a$. Note that $\sigma_{f, g, T^a}$ is absolutely continuous with respect to $\sigma_{f, T^a}$ (see, for example, Proposition 2.4 of \cite{Queffelec} for a proof). We claim that $\sigma_{f, T^a}$ is a continuous measure. Since $f \in \mathcal{Z}_1^\perp$ and $\sigma_{f, T}$ is a continuous measure, Wiener's lemma tells us that $\displaystyle{\lim_{H \to \infty}\frac{1}{H} \sum_{h=1}^H | \hat{\sigma}_{f, T}(h)|^2 = 0}$. Since $\displaystyle \left( \frac{1}{|a|H} \sum_{h=1}^{|a|H} | \hat{\sigma}_{f, T}(h) |^2\right)_H$ is a subsequence of $\displaystyle  \left( \frac{1}{H} \sum_{h=1}^{H} | \hat{\sigma}_{f, T}(h) |^2\right)_H$, we have
\[ \lim_{H \to \infty} \frac{1}{H} \sum_{h=1}^H | \hat{\sigma}_{f, T^a}(h) |^2 
= \lim_{H \to \infty} \frac{1}{H} \sum_{h=1}^{H} | \hat{\sigma}_{f, T}(ah) |^2 
\leq \lim_{H \to \infty} |a| \left( \frac{1}{|a|H} \sum_{h=1}^{|a|H} | \hat{\sigma}_{f, T}(h) |^2\right) = 0, \]
and again, by Wiener's lemma, $\sigma_{f, T^a}$ is a continuous measure. Hence, $\sigma_{f,g, T^a}$ is continuous, so we have
\begin{align}\label{CondExpSpec} 
0 &= \sigma_{f, g, T^a}(\left\{-\tau\right\}) = \lim_{H \to \infty} \frac{1}{H} \sum_{h=1}^H e^{-2\pi i h \tau}\int g(y) f(T^{ah}y) d\mu(y)\\
&\nonumber= \lim_{H \to \infty} \frac{1}{H} \sum_{h=1}^H  e^{-2\pi i h \tau}\expec_a(f \cdot f \circ T^{ah})(x) = \gamma_x(\left\{ -\tau \right\}) .\end{align}
Hence, $\gamma_x$ is a continuous measure. The proof is complete by letting $H \to \infty$ in $(\ref{condExpec-H})$.
\end{proof}
Here we introduce seminorms that are similar to the ones introduced in definition $\ref{HostKraSeminorms}$. These seminorms hold for any measure preserving system. 
\begin{mydef}
Suppose $(Y, \mathcal{Y}, \nu, U)$ is a measure preserving system, and $f \in L^\infty(\nu)$. We define seminorms $\| \la \cdot \ra \|_2$ and $\| \la \cdot \ra \|_3$ on $L^2(\nu)$ such that
\[ \| \la f \ra \|_2^4 = \lim_{H \to \infty} \frac{1}{H} \sum_{h=1}^H \left| \int f \cdot f \circ U^h \, d\nu \right|^2, \]
and
\[ \| \la f \ra \|_3^8 = \lim_{H \to \infty} \frac{1}{H} \sum_{h=1}^H \left\Vert \la f \cdot f \circ U^h \ra \right\Vert_2^4. \]
\end{mydef}
Certainly, if $U$ is ergodic, then $\| \la f \ra \|_k = \vertiii{ f }_k$ for $k = 2, 3$. We can easily verify that $\| \la \cdot \ra \|_2$ and $\| \la \cdot \ra \|_3$ are indeed seminorms. For example, we can show that $\| \la \cdot \ra \|_2$ is a positive semidefinite function by using the dominated convergence theorem and the pointwise convergence theorem:
\begin{align*}
\| \la f \ra \|_2^4 &= \lim_{H \to \infty}\frac{1}{H}\left| \int f \cdot f \circ U^h \, d\nu \right|^2\\
&= \lim_{H \to \infty}\frac{1}{H}\int \left(f \cdot f \circ U^h\right)(x) \, d\nu(x)\int{ \left(f \cdot f \circ U^h\right)(y)} \, d\nu(y)\\
&= \iint (f \otimes f)(x,y)d\mu \otimes \mu(x, y) \lim_{H \to \infty} \iint \frac{1}{H} \sum_{h=1}^H (f \otimes f)((U \otimes U)^h(x, y)) d\mu \otimes \mu(x, y) \\
&= \iint (f \otimes f)(x,y) \expec(f \otimes {f} | \mathcal{I}^2)(x, y) \, d\nu \otimes \nu(x, y) \\
&= \iint \expec(f \otimes {f} | \mathcal{I}^2)^2(x, y) \, d\nu \otimes \nu(x, y) \geq 0,
\end{align*}
where $\mathcal{I}^2$ is the sigma-algebra generated by $U \times U$-invariant sets.

Before we proceed to prove the theorem $\ref{generalCL}$, we will prove the following preliminary lemmas.
\begin{lemma}\label{singleFctn}
Suppose $(Y, \mathcal{Y}, \nu, U)$ is a measure preserving system, and $f \in L^\infty(\nu)$. Then
\begin{equation*}
\int \limsup_{N \to \infty} \left| \frac{1}{N} \sum_{n=1}^N f(U^ny)\right|^2 d\nu  \leq C \| \la f \ra \|_2^2
\end{equation*}
for $\nu$-a.e. $y \in Y$.
\end{lemma}

\begin{proof}
We denote $F_h(x) = f(x) f\circ U^h(x)$. We apply the inequality $(\ref{vcd-lim})$ by setting $a_n = f(U^ny)$ pointwise and the pointwise ergodic theorem to obtain
\begin{align}\label{vdc-mpt}
\limsup_{N \to \infty}\left| \frac{1}{N} \sum_{n=1}^N f(U^ny) \right|^2
&\leq \frac{C}{H} + \frac{C}{(H+1)^2} \sum_{h=1}^H (H + 1 - h) \left(\lim_{N \to \infty} \frac{1}{N}\sum_{n=0}^{N } F_h(U^ny) \right)\\
&= \frac{C}{H} + \frac{C}{(H+1)^2} \sum_{h=1}^H (H + 1 - h)\expec(f \cdot f \circ U^h | \mathcal{I})(y) \nonumber
\end{align}
where $\mathcal{I}$ is the sigma-algebra generated by $U$-invariant sets. Note that $\displaystyle \int \expec(f \cdot f \circ U^h | \mathcal{I}) \, d\nu = \int f \cdot f \circ U^h \, d\nu$. So if we take the integral on both sides of the inequality ($\ref{vdc-mpt}$), we would obtain the following after applying the Cauchy-Schwarz inequality to the second term:
\begin{equation}\label{U-IntegIneq}
\int \limsup_{N \to \infty} \left| \frac{1}{N} \sum_{n=1}^N f(U^ny) \right|^2 \, d\nu \leq \frac{C}{H} + C\left(\frac{1}{H} \sum_{h=1}^H \left| \int f \circ f \cdot U^h \, d\nu \right|^2 \right)^{1/2}.
\end{equation}
Now we let $H \to \infty$ to obtain
\begin{equation*}
\int \limsup_{N \to \infty} \left| \frac{1}{N} \sum_{n=1}^N f(U^ny) \right|^2 \, d\nu \leq C\left( \lim_{H \to \infty} \frac{1}{H} \sum_{h=1}^H \left| \int f \cdot f \circ U^h \, d\nu \right|^2\right)^{1/2} = C \| \la f \ra \|_2^2
\end{equation*}
\end{proof}

\begin{lemma}\label{doubleFctn}
Suppose $(X, \mathcal{F}, \mu, T)$ is an ergodic dynamical system, and $f_1, f_2 \in L^\infty(\mu)$. Then for any integers $a$ and $b$,
\[ \int \limsup_{N \to \infty} \left| \frac{1}{N} \sum_{n=1}^N f_1(T^{an}x) f_2(T^{bn}y) \right|^2 d\mu \otimes \mu(x, y) \leq C|a| \vertiii{ f_1 }_2^2 \]
for $\mu$-a.e. $x, y \in X$.
\end{lemma}

\begin{proof}
We denote $F_{1, h}(x) = f_1(x)f_1 \circ T^{ah}(x)$, and $F_{2, h}(x) = f_2(x)f_2\circ T^{bh}(x)$. If $U = T^a \times T^b$, then $(X^2, \mathcal{F} \otimes \mathcal{F}, \mu \otimes \mu, U)$ is a measure preserving system. Hence, we can use $(\ref{U-IntegIneq})$  in Lemma $\ref{singleFctn}$ to obtain, for any $H \in \NN$, 
\begin{align*} 
& \int \limsup_{N \to \infty} \left| \frac{1}{N} \sum_{n=1}^N f_1 \otimes f_2 (U^n(x, y)) \right|^2 d\mu \otimes \mu(x, y) \\
&\leq \frac{C}{H} + C \left(\frac{1}{H}\sum_{h=1}^H \left| \iint F_{1, h}(x)F_{2, h}(y) \, d\mu \otimes \mu(x, y) \right|^2 \right)^{1/2}\\
&\leq \frac{C}{H} + C \left(\frac{1}{H}\sum_{h=1}^H \left| \int f_1 \cdot f_1 \circ T^{ah} d\mu\right|^2 \right)^{1/2}.
\end{align*}
As we let $H \to \infty$, we obtain the desired result by lemma $\ref{HKPower}$.
\end{proof}

\begin{lemma}\label{doubleFctnWW}
Suppose $(X, \mathcal{F}, \mu, T)$ is an ergodic system, and $f_1, f_2 \in L^\infty(\mu)$. Then for any integers $a$ and $b$,
\[ \int \limsup_{N \to \infty} \sup_{t \in \RR} \left| \frac{1}{N} \sum_{n=1}^N f_1(T^{an}x) f_2(T^{bn}y) e^{2\pi i n t}\right|^2 d\mu \otimes \mu(x, y) \leq C|a|^{1/2}  \vertiii{ f_1 }_3^2 \]
for $\mu$-a.e. $x, y \in X$.
\end{lemma}

\begin{proof}
We denote $F_{1, h}(x) = f_1(x)f_1 \circ T^{ah}(x)$, and $F_{2, h}(x) = f_2(x)f_2\circ T^{bh}(x)$. By applying the inequality ($\ref{vcd-ww-lim}$) for $a_n = f_1(T^{an}x)f_2(T^{bn}y)$ pointwise, we obtain
\begin{align*}
\limsup_{N \to \infty} \sup_{t \in \RR} \left| \frac{1}{N} \sum_{n=1}^N f_1(T^{an}x)f_2(T^{bn}y) e^{2\pi i n t}\right|^2  &\leq \frac{C}{H} + \frac{C}{H} \sum_{h=1}^H \limsup_{N \to \infty}\left| \frac{1}{N} \sum_{n=1}^N F_{1, h}(T^{an}x) F_{2, h}(T^{bn}y) \right|.
\end{align*}
Again, if we set $U = T^a \times T^b$, then $(X^2, \mathcal{F} \otimes \mathcal{F}, \mu \otimes \mu, U)$ is a measure preserving system. Hence, Birkhoff's pointwise ergodic theorem asserts that the average $\displaystyle{ \frac{1}{N} \sum_{n=1}^N F_{1, h} \otimes F_{2, h}(U^n(x, y))}$ converges $\mu \otimes \mu$-a.e. By the Cauchy-Schwarz inequality, we have
\begin{align*}
&\int \limsup_{N \to \infty}\sup_{t \in \RR} \left| \frac{1}{N} \sum_{n=1}^N f_1(T^{an}x)f_2(T^{bn}y) e^{2\pi i n t}\right|^2 d\mu \otimes \mu(x, y)\\
&\leq \frac{C}{H} + \frac{C}{H} \sum_{h=1}^H \int \lim_{N \to \infty}\left| \frac{1}{N} \sum_{n=1}^N F_{1, h}(T^{an}x)F_{2, h}(T^{bn}y) \right| \, d\mu \otimes \mu(x, y)\\
&\leq \frac{C}{H} +  C\left(\frac{1}{H} \sum_{h=1}^H \int \lim_{N \to \infty}\left| \frac{1}{N} \sum_{n=1}^N F_{1, h}(T^{an}x)F_{2, h}(T^{bn}y) \right|^2 d\mu \otimes \mu(x, y) \right)^{1/2}.
\end{align*}
By lemma $\ref{doubleFctn}$, we know that 
\[\int \lim_{N \to \infty}\left| \frac{1}{N} \sum_{n=1}^N F_{1, h}(T^{an}x)F_{2, h}(T^{bn}y) \right|^2 d\mu \otimes \mu(x, y) \leq C|a| \vertiii{F_{1, h} }_2^2 = C|a| \vertiii{F_{1, h} }_2^2. \]
Hence,
\begin{align*} &\int \limsup_{N \to \infty} \sup_{t \in \RR} \left| \frac{1}{N} \sum_{n=1}^N f_1(T^{an}x) f_2(T^{bn}y) e^{2\pi i n t}\right|^2 d\mu \\
&\leq \frac{C}{H} + C \left(\frac{|a|}{H} \sum_{h=1}^H \vertiii{ f_1 \cdot f_1 \circ T^{ah} }_2^2 \right)^{1/2} \leq \frac{C}{H} + C \left( \frac{|a|^2}{H} \sum_{h=1}^H \vertiii{ f_1 \cdot f_1 \circ T^{ah} }_2^4 \right)^{1/4}.
\end{align*}
Let $H \to \infty$, and apply lemma $\ref{HKPower}$ to obtain the desired result.
\end{proof}
Now we are ready to prove the theorem $\ref{generalCL}$. The beginning of the proof is very similar to the second proof of the theorem $\ref{CL}$, where we apply the inequalities $(\ref{vcd-ww-lim})$ and $(\ref{vcd-lim})$, as well as the double recurrence theorem and the mean ergodic theorem. Then we use the integral kernel from the lemma $\ref{integralKernel}$ to obtain the integral expression for the upper bound, and then we use the inequality $(\ref{cubes-variant})$ and the lemma $\ref{doubleFctnWW}$ to obtain the desired result.

\begin{proof}[Proof of Theorem $\ref{generalCL}$]
We denote $F_{1, h}(x) = f_1(x)f_1 \circ T^{ah}(x)$, and $F_{2, h}(x) = f_2(x)f_2\circ T^{bh}(x)$. We apply the inequality ($\ref{vcd-ww-lim}$) by setting $a_n = f_1(T^{an}x)f_2(T^{bn}x)$, we obtain the following for all $H \in \NN$:
\begin{align*}
& \limsup_{N \to \infty} \sup_{t \in \RR} \left| \frac{1}{N} \sum_{n=1}^N f_1(T^{an}x)f_2(T^{bn}x) e^{2\pi i n t} \right|^2 \\
&\leq \frac{C}{H} + \frac{C}{H} \sum_{h = 1}^H \limsup_{N \to \infty} \left| \frac{1}{N} \sum_{n=1}^{N} F_{1, h}(T^{an}x) F_{2, h}(T^{bn}x) \right|\\
&\leq \frac{C}{H} + \left( \frac{C}{H}  \sum_{h = 1}^H \limsup_{N \to \infty} \left| \frac{1}{N} \sum_{n=1}^{N} F_{1, h}(T^{an}x) F_{2, h}(T^{bn}x) \right|^2 \right)^{1/2},
\end{align*}
where the second inequality is the consequence of the Cauchy-Schwarz inequality. Note that we can  apply the inequality $(\ref{vcd-lim})$ on the average $\displaystyle{ \left| \frac{1}{N} \sum_{n=1}^{N} F_{1, h}(T^{an}x) F_{2, h}(T^{bn}x) \right|^2}$ to obtain the following bound for $0 < K < N$:
\begin{align*}
& \limsup_{N \to \infty} \left| \frac{1}{N} \sum_{n=1}^{N} F_{1, h}(T^{an}x) F_{2, h} (T^{bn}x) \right|^2 \\
&\leq \frac{C}{K} + \frac{C}{(K+1)^2} \sum_{k=1}^K (K + 1 - k) \left( \limsup_{N \to \infty} \frac{1}{N} \sum_{n=1}^{N}\left( F_{1, h}\cdot F_{1, h} \circ T^{ak} \right)(T^{an}x)  \left( F_{2, h} \cdot F_{2, h} \circ T^{bk} \right) (T^{bn}x) \right).
\end{align*}
Note that the average
\[ \frac{1}{N} \sum_{n=1}^{N} \left( F_{1, h}\cdot F_{1, h} \circ T^{ak} \right)(T^{an}x) \cdot \left( F_{2, h} \cdot F_{2, h} \circ T^{bk} \right)(T^{bn}x) \]
converges as $N \to \infty$ by Bourgain's a.e. double recurrence theorem in \cite{BourgainDR}. Therefore,
\begin{align*}
& \int \limsup_{N \to \infty} \sup_{t \in \RR} \left| \frac{1}{N} \sum_{n=1}^N f_1(T^{an}x)f_2(T^{bn}x) e^{2\pi i n t} \right|^2 d\mu(x) \\
&\leq \frac{C}{H} + \frac{C}{H} \int \left( \sum_{h = 1}^H \limsup_{N \to \infty} \left|  \frac{1}{N} \sum_{n=1}^{N} F_{1, h}(T^{an}x) F_{2, h} (T^{bn}x) \right|^2 \right)^{1/2} d\mu(x)\\
&\leq \frac{C}{H} + \frac{C}{H}\left(  \sum_{h = 1}^H \int \limsup_{N \to \infty} \left|  \frac{1}{N} \sum_{n=1}^{N} F_{1, h}(T^{an}x) F_{2, h} (T^{bn}x) \right|^2 d\mu(x) \right)^{1/2} \mbox{ (by H\"older's inequality)}\\
&\leq \frac{C}{H} + \frac{C}{H}\left( \sum_{h=1}^H \left( \frac{C}{K} + \frac{C}{(K+1)^2} \sum_{k=1}^K (K + 1 - k) \right. \right. \\
&\left. \left. \left(  \lim_{N \to \infty}  \int \frac{1}{N} \sum_{n=1}^{N} \left(F_{1, h} \cdot F_{1, h} \circ T^{ak} \right)(T^{an}x)  \left(F_{2, h} \cdot F_{2, h} \circ T^{bk} \right) (T^{bn}x) \, d\mu(x) \right) \right) \right)^{1/2}.
\end{align*}
Since $T^a$ is a measure preserving transformation, we can apply the mean ergodic theorem to obtain
\begin{align*}
& \lim_{N \to \infty} \int \frac{1}{N} \sum_{n=1}^{N} \left(F_{1, h} \cdot F_{1, h} \circ T^{ak} \right)(T^{an}x)\left(F_{2, h} \cdot F_{2, h} \circ T^{bk} \right)(T^{bn}x)\, d\mu(x) \\
&= \lim_{N \to \infty} \int  \left(F_{1, h} \cdot F_{1, h} \circ T^{ak} \right)(x) \left( \frac{1}{N} \sum_{n=1}^{N}\left(F_{2, h} \cdot F_{2, h} \circ T^{bk} \right)(T^{(b-a)n}x) \right) \, d\mu(x) \\
&= \int \left(F_{1, h} \cdot F_{1, h} \circ T^{ak} \right)(x) \expec(F_{2, h} \cdot F_{2, h} \circ T^{bk} | \mathcal{I}_{b-a} )(x) d\mu(x),
\end{align*}
where $\mathcal{I}_{b-a}$ is the $\sigma$-algebra generated by $T^{b-a}$-invariant sets. By lemma $\ref{integralKernel}$, there exists a positive integer $l_{b-a}$ and partition $A_1, \ldots, A_{l_{b-a}}$ of $X$ such that
\[ \expec(F_{2, h} \cdot F_{2, h} \circ T^{bk} | \mathcal{I}_{b-a})(x) = \int \left(F_{2, h} \cdot F_{2, h} \circ T^{bk}\right)(y) K_{b-a}(x, y) d\mu(y), \]
where $\displaystyle{ K_{b-a}(x, y) = l_{b-a}\sum_{i=1}^{l_{b-a}} \ind_{A_i}(x) \ind_{A_i}(y) }$. Note that

\begin{align*}
&\iint \left(F_{1, h}(x) \cdot F_{1, h} \circ T^{ak} \right)(x)\left(F_{2, h} \cdot F_{2, h} \circ T^{bk} \right)(y) K_{b-a}(x, y)\, d\mu(x) d\mu(y) \\
&= \iint f_1 \otimes f_2(x, y) K_{b-a}(x, y)\, \left[ {f_1 \otimes f_2} (U^h(x, y)) \, {f_1 \otimes f_2}(U^k(x, y)) \, f_1 \otimes f_2(U^{h+k}(x, y)) \right]\,d\mu \otimes \mu(x, y),
\end{align*}
where $U = T^a \times T^b$ is a measure preserving transformation on $X^2$. Let $H = K$.
Note that, on the system $(X^2, \mu \otimes \mu, U)$, the inequality ($\ref{cubes-variant}$) tells us that we have
\begin{align}\label{AACubes} & \nonumber \frac{1}{H(H+1)^2} \left| \sum_{h, k = 0}^{H-1} (H + 1 - k) f_1 \otimes f_2 (U^h(x, y)) \, f_1 \otimes f_2 (U^k(x, y))  \, f_1 \otimes f_2 (U^{h+k}(x, y)) \right|^2 \nonumber \\
&\leq  \sup_{t \in \RR} \left| \frac{1}{H} \sum_{h=1}^{2(H-1)} f_1 \otimes f_2 (U^h(x, y)) e^{2\pi iht} \right|^2. \end{align}
By lemma $\ref{doubleFctnWW}$, we know that
\begin{align*}
& \int \limsup_{H \to \infty} \sup_{t \in \RR} \left| \frac{1}{H} \sum_{h=1}^H f_1 \otimes f_2 (U^h(x, y)) e^{2\pi iht} \right|^2 \, d\mu \otimes \mu(x, y) \\
&= \int \limsup_{H \to \infty} \sup_{t \in \RR} \left|^2 \frac{1}{H} \sum_{h=1}^H f_1(T^{ah}x)f_2 (T^{bh}y) e^{2\pi iht} \right| d\mu \otimes \mu(x, y) \\
&\leq C|a|^{1/2} \vertiii{ f_1 }_3^2 = 0.
\end{align*}
Hence, by letting $H \to \infty$, we obtain
\begin{align*}
& \int \limsup_{N \to \infty} \sup_{t \in \RR} \left| \frac{1}{N} \sum_{n=1}^N f_1(T^{an}x) f_2(T^{bn}x) e^{2\pi i n t} \right|^2 \, d\mu(x) \\
&\leq \iint f_1 \otimes f_2(x, y) K(x, y) \\
&\cdot \lim_{H \to \infty} \frac{1}{H(H+1)^2} \sum_{h, k = 0}^{H-1} (H + 1 - k) \left( {f_1 \otimes f_2} (U^h(x, y)) \, {f_1 \otimes f_2}(U^k(x, y)) \, f_1 \otimes f_2(U^{h+k}(x, y)) \right) \, d\mu \otimes \mu(x, y) \\
&= 0.
\end{align*}
\end{proof}

\section{Case when $f_1 \in \mathcal{A}_2^\perp$, $a = 1$, $b = 2$}\label{sec:maxIsom}
In this section, we will show that we can obtain a pointwise estimate to the Wiener-Wintner double recurrence averages using the seminorm of $\mathcal{A}_2$. This means that we can bound the double recurrence averages using the seminorm $N_2(\cdot)$ without taking the integral of the norm of the averages. This was not the case when we used the Host-Kra seminorm  $\vertiii{ \cdot }_3$, where we obtained the norm bound
\[ \int \limsup_{N \to \infty}\sup_{t \in \RR} \left| \frac{1}{N} \sum_{n=1}^N f_1(T^nx)f_2(T^{2n}x)e^{2\pi i n t} \right| d\mu \leq C \vertiii{ f_1}_3^2.\]

We recall that $(X, \mathcal{F}, \mu, T)$ is an ergodic system, $f_1, f_2 \in L^\infty(\mu)$ such that $\| f_i \|_\infty \leq 1$ for both $i = 1, 2$.

\begin{theorem}\label{maxIsom} Let $(X, \mathcal{F}, \mu, T)$ be an ergodic dynamical system. Then there exist a universal constant $C$ such that
\[ \limsup_{N \to \infty} \sup_{t \in \RR} \left| \frac{1}{N} \sum_{n=1}^N f_1(T^nx) f_2(T^{2n}x) e^{2\pi i n t} \right| \leq C [N_2(f_1)]^2 \]
for $\mu$-a.e. $x \in X$.
\end{theorem}

\begin{proof}
We first apply the inequality $(\ref{vcd-ww-lim})$ to the sequence $a_n = f_1(T^nx)f_2(T^{2n}x)$ pointwise to obtain
\begin{align*} \limsup_{N \to \infty}\sup_{t \in \RR} \left| \frac{1}{N} \sum_{n=1}^N f_1(T^nx) f_2(T^{2n}x)e^{2\pi i n t} \right|^2
&\leq \frac{C}{H} + \frac{C}{H} \sum_{h=1}^H\limsup_{N \to \infty}\left| \frac{1}{N}  \sum_{n=1}^{N } (f_1 \cdot f_1\circ T^h)(T^{n}x)(f_2 \cdot f_2\circ T^{2h})(T^{2n}x)\right|.
\end{align*}
Our main task is to show that
\begin{equation}\label{vcdlimit} \limsup_{H \to \infty} \frac{1}{H} \sum_{h=1}^H \lim_{N \to \infty} \left| \frac{1}{N} \sum_{n=1}^N \left(f_1 \cdot f_1 \circ T^h\right)(T^nx) \, \left(f_2 \cdot f_2 \circ T^{2h}\right)(T^{2n}x) \right| \leq [N_2(f_1)]^2
\end{equation}
for $\mu$-a.e. $x \in X$. We will first prove the following lemma, which would allow us to take conditional expectations of $f_i \cdot f_i \circ T^h$ for $i = 1, 2$ to the Kronecker factor $\mathcal{A}_1$, i.e. it would suffice to show that
\begin{equation*}\label{vcdlimit-cond} \limsup_{H \to \infty} \frac{1}{H} \sum_{h=1}^H \lim_{N \to \infty} \left| \frac{1}{N} \sum_{n=1}^N \expec\left(f_1 \cdot f_1 \circ T^h | \mathcal{A}_1 \right)(T^nx) \, \expec\left(f_2 \cdot f_2 \circ T^{2h} | \mathcal{A}_1\right)(T^{2n}x) \right| \leq [N_2(f_1)]^2.
\end{equation*}

\begin{lemma}\label{charFactor}
Suppose $F_1, F_2 \in L^{\infty}(X)$ such that $\| F_1 \|_\infty, \| F_2 \|_\infty \leq 1$. If $F_1 \in \mathcal{A}_1^\perp$, then for $\mu$-a.e. $x \in X$,
\[ \lim_{N \to \infty} \frac{1}{N} \sum_{n=1}^N F_1(T^nx) F_2(T^{2n}x) = 0. \]
Thus, $\mathcal{A}_1$ is a pointwise characteristic factor of this average, i.e. 
\begin{equation}\label{projection} \lim_{N \to \infty} \frac{1}{N} \sum_{n=1}^N F_1(T^nx) F_2(T^{2n}x) = \lim_{N \to \infty} \frac{1}{N} \sum_{n=1}^N \expec(F_1 | \mathcal{A}_1)(T^nx)\expec(F_2 | \mathcal{A}_1)(T^{2n}x). \end{equation}
\end{lemma}
%\textbf{Remark:} We know that the average above converges by Bourgain's double recurrence theorem. We will show that the norm limit of the average converges to $0$, hence it converges to $0$ pointwise as well. 
\begin{proof}
Since $\left| \frac{1}{N} \sum_{n=1}^N F_1(T^nx) F_2(T^{2n}x) \right|$ is non-negative, we can prove this lemma by showing
\[\int \limsup_{N \to \infty} \left| \frac{1}{N} \sum_{n=1}^N F_1(T^nx) F_2(T^{2n}x) \right| d\mu(x) = \lim_{N \to \infty} \int \left| \frac{1}{N} \sum_{n=1}^N F_1(T^nx) F_2(T^{2n}x) \right| d\mu(x) =  0\]
where the first equality holds by Bourgain's double recurrence theorem and Lebesgue's dominated convergence theorem. Note that Cauchy-Schwarz inequality asserts that
\[ \int \left| \frac{1}{N} \sum_{n=1}^N F_1(T^nx) F_2(T^{2n}x) \right| d\mu(x) \leq \left( \int \left| \frac{1}{N} \sum_{n=1}^N F_1(T^nx) F_2(T^{2n}x) \right|^2 d\mu(x) \right)^{1/2}. \]
We will proceed by applying the inequality $(\ref{vcd-lim})$ to $a_n = F_1(T^nx)F_2(T^{2n}x)$ pointwise. Observe that, by the inequality $(\ref{revFatou})$, we have
\begin{align*}
& \limsup_{N \to \infty}\int \left| \frac{1}{N} \sum_{n=1}^N F_1(T^nx) F_2(T^{2n}x) \right|^2 d\mu(x) \\
&\leq \int \limsup_{N \to \infty} \left| \frac{1}{N} \sum_{n=1}^N F_1(T^nx) F_2(T^{2n}x) \right|^2 d\mu(x) \\
&\leq \frac{C}{H} + \frac{C}{(H+1)^2} \sum_{h=1}^H (H + 1 - h) \int \limsup_{N \to \infty} \frac{1}{N} \sum_{n=1}^N (F_1 \cdot F_1 \circ T^h)(T^nx) (F_2 \cdot F_2 \circ T^{2h})(T^{2n}x) d\mu \\
&\leq \frac{C}{H} + \frac{C}{H} \sum_{h=1}^H \left| \int \limsup_{N \to \infty} \frac{1}{N} \sum_{n=1}^N (F_1 \cdot F_1 \circ T^h)(T^nx) (F_2 \cdot F_2 \circ T^{2h})(T^{2n}x) d\mu \right|\\
\end{align*}
Note that the limit inside the integral exists by the double recurrence theorem. Hence, the dominated convergence theorem tells us
\begin{align*}
& \limsup_{N \to \infty}\int \left| \frac{1}{N} \sum_{n=1}^N F_1(T^nx) F_2(T^{2n}x) \right|^2 d\mu(x) \\
&\leq  \frac{C}{H} + \frac{C}{H} \sum_{h=1}^H \left| \lim_{N \to \infty} \int  \frac{1}{N} \sum_{n=1}^N (F_1 \cdot F_1 \circ T^h)(T^nx) (F_2 \cdot F_2 \circ T^{2h})(T^{2n}x) d\mu \right|\\
&=  \frac{C}{H} + \frac{C}{H} \sum_{h=1}^H \left|\lim_{N \to \infty} \int  (F_1 \cdot F_1 \circ T^h)(x) \frac{1}{N} \sum_{n=1}^N (F_2 \cdot F_2 \circ T^{2h})(T^{n}x) d\mu \right|\end{align*}
Then we use the mean ergodic theorem and the Cauchy-Schwarz Inequality to obtain
\begin{align*}
\limsup_{N \to \infty}\int \left| \frac{1}{N} \sum_{n=1}^N F_1(T^nx) F_2(T^{2n}x) \right|^2 d\mu(x) 
&\leq  \frac{C}{H} + \frac{C}{H} \sum_{h=1}^H \| F_2 \|_\infty^2 \left| \int  (F_1 \cdot F_1 \circ T^h)(x) d\mu \right|\\
&\leq  \frac{C}{H} + C \left(\frac{1}{H} \sum_{h=1}^H \left| \int  (F_1 \cdot F_1 \circ T^h)(x) d\mu \right|^2 \right)^{1/2}\\
&=  \frac{C}{H} + C \left(\frac{1}{H} \sum_{h=1}^H \left| \hat{\sigma}_{F_1}(h)\right|^2 \right)^{1/2},
\end{align*}
where $\sigma_{F_1}$ is the spectral measure of $F_1$ with respect to the transformation $T$. Now we let $H \to \infty$ to obtain
\[ \limsup_{N \to \infty} \int \frac{1}{N} \sum_{n=1}^N \left| F_1(T^nx)F_2(T^{2n}x) \right|^2 d\mu(x)  \leq C\left( \lim_{H \to \infty} \frac{1}{H} \sum_{h=1}^H |\hat{\sigma}_{F_1}(h)|^2 \right)^{1/2}, \]
and because $F_1 \in \mathcal{A}_1^\perp$, the spectral measure $\sigma_{F_1}$ is continuous, so the Wiener's lemma implies the right hand side of above limit equals $0$.
\end{proof}
Now we will conclude the proof of Theorem $\ref{maxIsom}$. Set $F_{1, h} = f_1 \cdot f_1 \circ T^h$, and $F_{2, h} = f_2 \cdot f_2 \circ T^{2h}$. Denote
\[ P_N(F_{1, h}, F_{2, h}) = \frac{1}{N} \sum_{n=1}^N \expec(F_{1, h} | \mathcal{A}_1)\circ T^n \expec(F_{2, h} | \mathcal{A}_1) \circ T^{2n}. \]
By $(\ref{projection})$, it suffices to prove that $\displaystyle \limsup_{H \to \infty} \frac{1}{H} \sum_{h=1}^H \limsup_{N \to \infty}P_N(F_{1, h}, F_{2, h})(x) \leq [N_2(f_1)]^2$ for $\mu$-a.e. $x \in X$ in order to show $(\ref{vcdlimit})$.  Let $\left\{ e_j \right\}$ be an eigenbasis of $\mathcal{A}_1$, where $\lambda_j$ is a corresponding eigenvalue of $e_j$. Then we would have 
\[ \expec(F_{1, h} | \mathcal{A}_1)\circ T^n = \sum_{j=0}^\infty \left(\int F_{1, h} {e_j} \, d\mu \right) \lambda_j^n e_j \mbox{ and } \expec(F_{2, h} | \mathcal{A}_1)\circ T^{2n} = \sum_{l=0}^\infty  \left( \int F_{2, h} {e_l} \, d\mu \right) \lambda_l^{2n} e_l \] 
in the $L^2$-norm. Hence,
\[ \lim_{N \to \infty} P_N(F_{1, h}, F_{2, h}) = \lim_{N \to \infty} \frac{1}{N} \sum_{n=1}^N \sum_{j = 1}^\infty \sum_{l = 1}^\infty  \left( \int F_{1, h} \overline{e_j} \, d\mu \right) \left( \int F_{2, h}\overline{e_l} \, d\mu \right) \lambda_j^n \lambda_l^{2n} e_j e_l \]
in the $L^2$-norm. Note that for each $j$ and $l$, 
\[ \lim_{N \to \infty} \frac{1}{N} \sum_{n=1}^N \lambda_j^n \lambda_l^{2n} = \left\{ \begin{array}{ll} 
1 & \mbox{if } \lambda_j = \overline{\lambda}_l^2 \\
0 & \mbox{otherwise.} \end{array} \right. \]
Hence, if we denote $R = \left\{ (j, l_j) \in \NN^2 :  \lambda_j = \overline{\lambda}_{l_j}^2 \right\}$, then
\[ \lim_{N \to \infty} P_N(F_{1, h}, F_{2, h}) = 
\sum_{(j, l_j) \in R} \left( \int F_1 \overline{e_j} \, d\mu \right) \left( \int F_{2, h} \overline{e_{l_j}} \, d\mu \right) e_j e_{l_j} \]
in the $L^2$-norm. Note that the sequence 
\[ B_J = \left( \sum_{(j, l_j) \in R, j \leq J} \left( \int F_{1, h} \overline{e_j} \, d\mu \right) \left( \int F_{2, h} \overline{e_{l_j}} \, d\mu \right) e_j e_{l_j} \right)_J\]
converges to $\displaystyle{ \lim_{N \to \infty} P_N(F_{1, h}, F_{2, h})}$ in the $L^2$-norm as $J \to \infty$. Therefore, there exists a subsequence $(B_{J_k})_k$ that converges to $\displaystyle{ \lim_{N \to \infty} P_N(F_{1, h}, F_{2, h})(x)}$ for $\mu$-a.e. $x \in X$. Thus,
\begin{align*}
& \lim_{N \to \infty} \frac{1}{N} \sum_{n=1}^N \expec(F_{1, h} | \mathcal{A}_1)(T^nx) \expec(F_{2, h} | \mathcal{A}_1)(T^{2n}x) \\
&= \lim_{k \to \infty} \sum_{(j, l_j) \in R , j \leq J_k} \left( \int F_{1, h} \overline{e_j} \, d\mu \right) \left( \int F_{2, h} \overline{e_{l_j}} \, d\mu \right) e_j(x) e_{l_j}(x) \\
&\leq \lim_{k \to \infty} \left( \sum_{(j, l_j) \in R , j \leq J_k} \left| \int F_{1, h} \overline{e_j}\, d\mu \right|^2 \right)^{1/2} \left( \sum_{(j, l_j) \in R , j \leq J_k} \left| \int F_{2, h} \overline{e_{l_j}}\, d\mu \right|^2 \right)^{1/2} \mbox{ (Cauchy-Schwarz Inequality)}
\end{align*}
\begin{align*}
&\leq \left( \sum_{j = 1}^\infty \left| \int F_{1, h} \overline{\ e_j} \, d\mu \right|^2 \right)^{1/2} \left( \sum_{l = 1}^\infty \left| \int F_{2, h} \overline{e_{l}} \, d\mu \right|^2 \right)^{1/2} \\
&= \| \expec(F_{1, h} | \mathcal{A}_1) \|_2 \| \expec(F_{2, h} | \mathcal{A}_1) \|_2\\
&\leq \| \expec(F_{1, h} | \mathcal{A}_1) \|_2,
\end{align*}
since $\| f_2  \|_\infty \leq 1$. Therefore,
\begin{align*} 
& \limsup_{H \to \infty}\frac{1}{H} \sum_{h=1}^H \lim_{N \to \infty} \frac{1}{N} \sum_{n=1}^N \left(f_1 \cdot f_1 \circ T^h\right)(T^nx) \left(f_2 \cdot f_2 \circ T^{2h}\right)(T^{2n}x) \\
&\leq\limsup_{H \to \infty} \frac{1}{H} \sum_{h=1}^H \| \expec(f_1 \cdot f_1 \circ T^h | \mathcal{A}_1) \|_2 \\
&\leq \left( \limsup_{H \to \infty} \frac{1}{H} \sum_{h=1}^H \| \expec(f_1 \cdot f_1 \circ T^h | \mathcal{A}_1) \|_2^2 \right)^{1/2}
= [N_2(f_1)]^2,
\end{align*}
where the second inequality holds by the Cauchy-Schwarz inequality.
\end{proof}

\section{Case when both $f_1, f_2 \in \mathcal{Z}_2$}
\label{sec:CVinZ2}
Here we prove the convergence of double recurrence Wiener Wintner averages for the case where $f_1, f_2 \in \mathcal{Z}_2$. To do so, we will use the structural properties of nilsystems, which we shall discuss briefly.

Let $(X, \mathcal{F}, \mu, T)$ be an ergodic system. Recall that $X$ is called a \textit{k-step nilsystem} if $X$ is a homogeneous space of a $k$-step nilpotent Lie group $G$ (such a manifold is called a \textit{nilmanifold}). Let $\Lambda$ be a discrete cocompact subgroup of $G$ such that $X = G / \Lambda$. The outline of the proof of the following theorem is given in \cite{HostKraCubes}.
\begin{theorem}[B. Host, B. Kra \cite{HostKraCubes}]\label{CL-nil}
If $X$ is a Conze-Lesigne system, then it is the inverse limit of a sequence of $2$-step nilsystems.
\end{theorem}
In the outline of the proof, $X$ is reduced to the case where $X$ is a group extension of the Kronecker factor $Z_1$ and torus $U$, with cocycle $\rho: Z_1 \to U$. A group $G$ is defined to be a family of transformations of $X = Z_1 \times U$, where $U$ is a finite dimensional torus and $Z_1$ is the Kronecker factor of $X$ that has the structure of compact abelian Lie group. If $g \in G$, $(z, u) \in X$, then
\[g\cdot(z, u) = (sz, uf(z))\]
where $s \in Z_1$ and $f: Z_1 \to U$ satisfy the Conze-Lesigne equation
\[ \rho(sz)\rho(z)^{-1} = f(Rz)f(z)^{-1}c \]
for some constant $c \in U$. It can be easily verified that $G$ is a $2$-step nilpotent group, and $T$ corresponds to $(\beta, \rho) \in G$, where $\beta \in Z_1$ such that if $\pi_1: Z_2 \to Z_1$ is a factor map, then $\pi_1(Tx) = \beta\pi_1(x)$. Furthermore, if $G$ is given a topology of convergence in probability, then we know that $G$ is a Lie group.\\[5pt]
The outline of the proof given in \cite{HostKraCubes} concludes by stating that $G$ acts on $X$ transitively, and $X$ can be identified with the nilmanifold $G/\Lambda$, where $\Lambda$ is a stabilizer group of a point $x_0 \in X$ (hence it is a discrete cocompact subgroup of $G$). Furthermore, $\mu$ is a Haar measure on $X$, and $T$ is a translation by the element $(\beta, \rho) \in G$. Hence, $T$ acts on $X$ by translation. We will use this fact to prove the convergence of the double recurrence Wiener Wintner average for the case when $f_1, f_2 \in \mathcal{Z}_2$.\\[5pt]
The following convergence result of Leibman will be used. We say $\left\{g(n)\right\}_{n \in \ZZ}$ is a \textit{polynomial sequence} if $g(n) = a_1^{p_1(n)} \cdots a_m^{p_m(n)}$, where $a_1, \ldots, a_m \in G$, and $p_1, \ldots, p_m$ are polynomials taking on integer value on the integers.

\begin{theorem}[A. Leibman \cite{Leibman}]\label{Leibman}Let $X = G/\Lambda$ be a nilmanifold and $\left\{ g(n) \right\}_{n \in \ZZ}$ be a polynomial sequence in $G$. Then for any $x \in X$ and continuous function $F$ on $X$, the average
\[ \frac{1}{N} \sum_{n=1}^N F(g(n)x) \]
converges as $N \to \infty$.
\end{theorem} 

\begin{theorem}\label{BothInZ2}
Let $(X, \mathcal{F}, \mu, T)$ be an ergodic dynamical system. Suppose $f_1, f_2 \in \mathcal{Z}_2$ are both continuous functions on $X$. Then the average
\[ \frac{1}{N} \sum_{n=1}^N f_1(T^{an}x)f_2(T^{bn}x)e^{2\pi i n t} \]
converges, as $N \to \infty$, off of a single null-set that is independent of $t$.
\end{theorem}

\begin{proof}
In this proof, we will consider two cases: The case when $t$ is rational, and the case when $t$ is irrational.\\[5pt]
\textbf{Case I: When $t$ is rational.} 
Fix $t \in \QQ$. Let $S_t$ be a rotation on $\mathbb{T}$ by $e^{2\pi i t}$. Let $(X \times \mathbb{T}, \mu \otimes m, U)$ be a measure preserving system, where $m$ is the Lebesgue measure on $\mathbb{T}$, and $U = T \otimes S_t$. Define $F_1(x, y) = f_1(x)e^{2\pi i \alpha_1y}$, and $F_2(x, y) = f_2(x)e^{2\pi i \alpha_2y}$, where $\alpha_1, \alpha_2 \in \RR$ such that $\alpha_1 a + \alpha_2 b = 1$. Then
\begin{equation}\label{rational} \frac{1}{N} \sum_{n=1}^N F_1(U^{an}(x, y))F_2(U^{bn}(x, y)) = \frac{e^{2\pi i y}}{N} \sum_{n=1}^N f_1(T^{an}x)f_2(T^{bn})e^{2\pi i n t}\end{equation}
Note that the average on the left hand side of $(\ref{rational})$ converges $\mu \otimes m$-a.e. as $N \to \infty$ by Bourgain's double recurrence theorem \cite{BourgainDR}. So there exists a set of full measure $V_t \subset X \times \mathbb{T}$ such that the average in $(\ref{rational})$ converges for all $(x, y) \in V_t$. If $V = \bigcup_{t \in \QQ} V_t$, then $V$ is a set of full measures such that the average on $(\ref{rational})$ converges for all $(x, y) \in V$ for all $t \in \QQ$. This implies that the claim holds for $\mu$-a.e. $x \in X$ when $t \in \QQ$. \\[5pt]
\textbf{Case II: When $t$ is irrational.}
Without loss of generality, we let $X = Z_2$, the Conze-Lesigne system. Let $\beta \in Z_1$ is an element such that for any $(z, u) \in Z_1 \times U = X$, $T(z, u) = (\beta z, u\rho(z))$. In other words, $T$ acts on $Z_1$ as a rotation by $\beta$ (here, we let $Z_1$ be a multiplicative abelian group). Then note that $B = \la \beta \ra$, the cyclic subgroup generated by $\beta$, is dense in the Kronecker factor $Z_1$. Define a character $\phi_t: B \to \mathbb{T}$ such that $\phi_t(\beta) = e^{2\pi i t}$. Such group homomorphism exists since $t$ is irrational, and $\la e^{2\pi i t} \ra$ generates a dense cyclic subgroup in $\mathbb{T}$. \\[5pt]
We claim that there exists a multiplicative character $\bar{\phi_t}: Z_1 \to \mathbb{T}$ such that $\bar{\phi_t}|_{B} = \phi_t$. Since $B$ is dense in $Z_1$, for any $z \in Z_1$, there exists a sequence $(\beta^{n_k})_k$ such that $\lim_{k \to \infty} \beta^{n_k} = z$. So we define
\[ \bar{\phi_t}(z) = \lim_{k \to \infty} \phi_t(\beta)^{n_k}. \]
We must show that this limit converges, which would show that $\overline{\phi}_t$ is well-defined by the continuity of $\phi$. Note that $\mathbb{T}$ is compact, so there exists a converging subsequence $(\phi_t(\beta)^{n_{k_l}}) \in \mathbb{T}$ such that $\lim_{l \to \infty} \phi_t(\beta)^{n_{k_l}} = \gamma$ for some $\gamma \in \mathbb{T}$. We will show that $\lim_{k \to \infty} \phi_t(\beta)^{n_k} = \gamma$. 
Assuming on the contrary, suppose that there exists a subsequence $(\phi_t(\beta)^{n_{k_m}})_m$ such that $|\phi_t(\beta)^{n_{k_m}} - \gamma| > \epsilon$ for all $m \in \NN$. This implies that, for sufficiently large $l$, we have $|\phi_t(\beta)^{n_{k_m}} - \phi_t(\beta)^{n_{k_l}} | > \epsilon / 2$. This however contradicts the continuity of $\phi_t$, since if $d_{Z_1}$ is the metric on $Z_1$, then $d_{Z_1}(\beta^{n_{k_l}}, \beta^{n_{k_m}}) \to 0$ as $l, m \to \infty$, because both $\beta^{n_{k_l}}$ and $\beta^{n_{k_m}}$ converges to the same limit $z$. This proves that $\bar{\phi_t}$ is well-defined for all $z \in Z_1$. The fact that $\bar{\phi_t}$ is a multiplicative character is obvious from the way $\bar{\phi_t}$ is defined in terms of $\phi_t$.\\[5pt]
We define a continuous function $f_t := \bar{\phi}_t \circ \pi_1$, where $\pi_1: Z_2 \to Z_1$ is the factor map. We note that
\[ f_t(T^nx) = \bar{\phi}_t (\pi_1(T^nx)) = \bar{\phi}_t (\pi_1(x)\beta^n) = f_t(x)\phi_t(\beta)^n = f_t(x)e^{2\pi i n t}. \]
Therefore,
\[ \frac{1}{N} \sum_{n=1}^N f_1(T^{an}x)f_2(T^{bn}x)f_t(T^nx) = \frac{f_t(x)}{N} \sum_{n=1}^N f_1(T^{an}x)f_2(T^{bn}x)e^{2\pi i n t}. \]
To show the convergence of this average, let $F(x_1, x_2, x_3) = f_1(x_1)f_2(x_2)f_t(x_3)$ be a function on $X^3 = G^3/\Lambda^3$. Let $T_1 = T \times \Id \times \Id$, $T_2 = \Id \times T \times \Id$, and $T_3 = \Id \times \Id \times T$. Note that an action of $T_1$ on $X^3$ corresponds to $g_1 = ((\beta, \rho), e, e) \in G^3$ (where $e$ is the identity element of $G$), and similarly, $T_2$ corresponds to $g_2 = (e, (\beta, \rho), e) \in G^3$, and $T_3$ corresponds to $g_3 = (e, e, (\beta, \rho)) \in G^3$. Thus,
\[ g(n) = g_1^{an} g_2^{bn} g_3^{n} \]
is a polynomial sequence. Furthermore, if $\vec{x} = (x, x, x) \in X^3$, then
\[ \frac{1}{N} \sum_{n=1}^N F(g(n)\vec{x}) = \frac{1}{N} \sum_{n=1}^N f_1(T^{an}x)f_2(T^{bn}x)f_t(T^nx) \]
converges by theorem $\ref{Leibman}$. 
\end{proof}

\section*{Appendix}
Here we provide the proofs of the inequalities mentioned in Section $\ref{sec:Inequalities}$.

\begin{proof}[Proof of Lemma $\ref{lem-vdc}$]
One can find the proof of van der Corput's lemma in many different sources; see \cite{KuipersNiederreiter} for example.
\end{proof}
\begin{proof}[Proof of Lemma $\ref{lem-vcd-ww}$]
To show $(\ref{vcd-lim})$, we take the limit supremum (as $N \to \infty$) on both sides of $(\ref{vdc})$. Then we obtain
\[ \limsup_{N \to \infty} \left| \frac{1}{N} \sum_{n=0}^{N-1} a_n \right|^2 \leq \frac{1}{H} + \frac{2}{(H+1)^2} \sum_{h=1}^H (H + 1 - h) \Re \left( \limsup_{N \to \infty} \frac{1}{N} \sum_{n=0}^{N - h - 1} a_n \overline{a}_{n+h} \right) \]
Let $u_n$ be another sequence of complex numbers norms bounded by $1$. Then, for fixed $h$, we have
\[ \frac{1}{N} \sum_{n=0}^{N - h - 1} u_n =  \frac{1}{N} \sum_{n=0}^{N} u_n - \frac{1}{N} \sum_{n=N-h}^{N} u_n. \]
Since $|u_n| \leq 1$, we know that for fixed $h$,
\[ \limsup_{N \to \infty} \left| \frac{1}{N} \sum_{n=N-h}^{N} u_n \right| \leq \limsup_{N \to \infty} \frac{h}{N} = 0. \]
Therefore, 
\begin{equation}\label{sameLimit} \limsup_{N \to \infty} \frac{1}{N} \sum_{n=0}^{N - h - 1} u_n = \limsup_{N \to \infty} \frac{1}{N} \sum_{n=0}^{N} u_n \end{equation}
Now apply $(\ref{sameLimit})$ to $u_n = a_n \overline{a}_{n+h}$, we obtain
\[\limsup_{N \to \infty} \left| \frac{1}{N} \sum_{n=0}^{N-1} a_n \right|^2 \leq \frac{1}{H} + \frac{2}{(H+1)^2} \sum_{h=1}^H (H + 1 - h) \Re \left( \limsup_{N \to \infty} \frac{1}{N} \sum_{n=0}^{N} a_n \overline{a}_{n+h} \right), \]
so set $C > 2$, and the claim holds.

To show $(\ref{vcd-ww})$, we recall that corollary 2.1 of \cite{AssaniWWET} states that
\[ \sup_{t \in \RR} \left| \frac{1}{N} \sum_{n=1}^N a_n e^{2\pi i n t} \right|^2 \leq \frac{2}{NH} \sum_{n=1}^{N} |a_n|^2 + \frac{4}{H} \sum_{h=1}^H \left| \frac{1}{N} \sum_{n=1}^N a_n \overline{a}_{n+h} \right|. \]
Since $\sup_n |a_n|^2 \leq 1$, we have
\[ \frac{2}{NH} \sum_{n=1}^{N} |a_n|^2 \leq \frac{2}{H}. \]
Choose $C > 4$, and we obtain the desired inequality.

To show $(\ref{vcd-ww-lim})$, we apply limit supremum (as $N \to \infty$) to both sides of $(\ref{vcd-ww})$, which gives us
\[ \limsup_{N \to \infty} \sup_{t \in \RR} \left| \frac{1}{N} \sum_{n=1}^N a_ne^{2\pi i n t} \right|^2 \leq \frac{C}{H} + \frac{C}{H} \sum_{h=1}^H \limsup_{N \to \infty}\left| \frac{1}{N} \sum_{n-1}^{N} a_n \overline{a}_{n+h} \right|.  \]
We apply $(\ref{sameLimit})$ to $u_n = a_n \overline{a}_{n+h}$, and we obtain the desired inequality.
\end{proof}

\begin{proof}[Proof of Lemma $\ref{lem-revFatou}$]
Note that
\[ \limsup_{n \to \infty} f_n = \inf_{k \in \NN} \sup_{n \geq k} f_n. \]
So if we set $\displaystyle g_k = \sup_{n \geq k} f_n$, and since $g_k$ is decreasing, its limit exists pointwise (i.e. $\displaystyle \limsup_n f_n$), and $\displaystyle g_1 = \sup_{n \geq 1} f_n \leq F$, we apply the dominated convergence theorem to obtain the following:
\[ \int \limsup_{n \to \infty} f_n \, d\mu = \lim_{k \to \infty} \int g_k \, d\mu = \lim_{k \to \infty} \int \sup_{n \geq k} f_n \, d\mu.\]
Of course, $\displaystyle f_i \leq \sup_{n \geq k} f_n$ for all $i \geq k$, we know that $\displaystyle \int f_i d\mu \leq \int \sup_{n \geq k} f_n d\mu$. So in particular, $\displaystyle \sup_{n \geq k} \int f_n \leq \int \sup_{n \geq k} f_n$. Hence,
\[ \int \limsup_{n \to \infty} f_n \, d\mu  \geq \lim_{k \to \infty} \sup_{n \geq k} \int f_n = \limsup_{n \to \infty} \int f_n \, d\mu\]
\end{proof}

\begin{proof}[Proof of Lemma $\ref{lem-cubes-variant}$]
This proof is a small modification of the proof provided in Lemma 5 of \cite{AssaniCubes}. By the Cauchy-Schwarz inequality, we have
\begin{align*}
& \left| \frac{1}{H(H+1)^2} \sum_{h, k = 0}^{H-1} (H + 1 - k) a_h b_k c_{h+k} \right|^2 \\
&\leq \| a \|_\infty^2 \left( \frac{1}{H} \sum_{h=0}^{H-1} \left| \frac{1}{(H+1)^2} \sum_{k=0}^{H-1}(H + 1 - k) b_k c_{h + k} \right|^2 \right).
\end{align*}
Set $\displaystyle B_k =  b_k\frac{(H + 1 - k)}{H + 1}$, and the inequality above becomes
\begin{align*}
& \left| \frac{1}{H(H+1)^2} \sum_{h, k = 0}^{H-1} (H + 1 - k) a_h b_k c_{h+k} \right|^2 \\
&\leq \| a \|_\infty^2 \left( \frac{1}{H} \sum_{h=0}^{H-1} \left| \frac{1}{H+1} \sum_{k=0}^{H-1} B_k c_{h+k} \right|^2 \right) \\
& \leq \| a \|_\infty^2 \frac{1}{H} \sum_{h=0}^{H-1} \left| \int \left( \sum_{k = 0}^{H-1} B_k e^{-2\pi i k t} \right) \left(\frac{1}{H+1} \sum_{k=0}^{2(H-1)} c_k e^{2\pi i k t} \right) e^{-2\pi i h t} dt \right|^2.
\end{align*}
We apply Parseval's inequality to the integral above to obtain
\begin{align*}
& \left| \frac{1}{H(H+1)^2} \sum_{h, k = 0}^{H-1} (H + 1 - k) a_h b_k c_{h+k} \right|^2 \\
&\leq \| a \|_\infty^2 \frac{1}{H} \sum_{h=0}^{H-1} \int \left| \sum_{k=0}^{H-1} B_k e^{-2\pi i k t} \right|^2 \left| \frac{1}{H+1} \sum_{k'=0}^{2(H-1)} c_{k'} e^{2\pi i k' t} \right|^2 dt \\
&\leq \| a \|_\infty^2  \sup_{t \in \RR} \left| \frac{1}{H+1} \sum_{k=0}^{2(H-1)} c_k e^{2\pi i k t} \right|^2 \frac{1}{H} \sum_{h=0}^{H-1}\int \left| \sum_{k=0}^{H-1} B_k e^{-2\pi int} \right|^2 dt \\
&\leq \|a \|_\infty^2 \sup_{t \in \RR} \left| \frac{1}{H+1} \sum_{k=0}^{2(H-1)} c_k e^{2\pi i k t} \right|^2 \frac{1}{H+1} \sum_{h=0}^{H-1} |B_k|^2.
\end{align*}
Since $|B_k| < 1$, we know that $\displaystyle \frac{1}{H+1} \sum_{h=0}^{H-1} |B_k|^2 \leq 1$. Thus, $(\ref{cubes-variant})$ holds.
\end{proof}

\textbf{Remark:} The first and the third authors are currently preparing the extension of theorem $\ref{MainResult}$ to show that the sequence $u_n = f_1(T^{an}x)f_2(T^{bn}x)$ is a good universal weight for the pointwise ergodic theorem \cite{AssaniMooreUnivWeight}.
\bibliography{WWDR_preprint}
\bibliographystyle{plain}

\end{document}